\documentclass[11pt,a4paper,twoside]{article}

\usepackage[applemac]{inputenc}
\usepackage[T1]{fontenc}
\usepackage[english]{babel}

\usepackage{fancyhdr}
\usepackage{indentfirst} 
\usepackage{newlfont}
\usepackage{verbatim}
\usepackage{enumerate}
\usepackage{hyperref}

\usepackage{amsthm}  
\usepackage{amssymb}  
\usepackage{amscd}
\usepackage{amsmath} 
\usepackage{latexsym} 
\usepackage{amsfonts}
\usepackage{amsbsy}   
\usepackage{mathrsfs} 
\usepackage{arcs}
\usepackage{yhmath}

\usepackage{graphicx}
\usepackage{color}
\usepackage{subfigure}

\theoremstyle{plain}
\newtheorem{teo}{Theorem}

\newtheorem{prop}[teo]{Proposition}
\newtheorem{cor}[teo]{Corollary}

\theoremstyle{definition}

\theoremstyle{remark}
\newtheorem{oss}[teo]{Remark}

\newcommand{\rp}[1]{\ensuremath{\mathbb{RP}^{#1}}}
\newcommand{\C}{\ensuremath{\mathbb{C}}}
\newcommand{\s}[1]{\ensuremath{\mathbf{S}^{#1}}}
\newcommand{\z}[1]{\ensuremath{\mathbb{Z}_{#1}}}

\begin{document}

\title{Diffeomorphic vs isotopic links in lens spaces}


\author{Alessia Cattabriga, Enrico Manfredi}

\maketitle

\begin{abstract}
Links in lens spaces may be defined to be equivalent by ambient isotopy or by diffeomorphism of pairs. In the first case, for all the combinatorial representations of links,  there is a set of Reidemeister-type moves on diagrams  connecting isotopy equivalent links. In this paper we provide a set of moves on disk, band and grid  diagrams that connects diffeo equivalent links: there are up to four isotopy equivalent links in each diffeo equivalence class. Moreover, we investigate how the diffeo equivalence relates to the lift of the link in the $3$-sphere: in the particular case of oriented primitive-homologous knots, the lift completely determines the knot class in $L(p,q)$ up to diffeo equivalence, and thus only four possible knots up to isotopy equivalence can have the same lift.
 \\
\\ {{\it Mathematics Subject
Classification 2010:} Primary 57M27, 57M10; Secondary 57M25.}\\

\end{abstract}

\begin{section}{Introduction}

\paragraph{Interest of links in lens spaces}
Lens spaces are an infinite family of closed orientable 3-manifolds obtained as cyclic quotients of $\s3$. They were  introduced in  \cite{T} at the beginning of the 20th century   and, since then,  a rich theory of invariants and representations for this manifolds has been developed and is now well established. So, in the last decades,  the research focused on  the study of links, i.e. closed 1-submanifold, embedded   in these manifolds:   the techniques of representation developed for lens spaces   produced  different notions of diagrams for the links inside them, establishing a connection  with the widespread theory of  links  in $\s3$ and  determining the possibility of extending a lot of classical invariants (see for example \cite{BGH, CMM, Co, Dr, GM, GM2, G1,G2,HP}).

The methods developed and the results achieved have revealed to be useful also to study   problems not directly formulated in terms of links in lens spaces. Recently, in \cite{BGH,He} important steps toward the proof of the long standing Berge conjecture (characterizing  the class of knots in $\s3$ admitting  lens spaces surgery) have been done by rephrasing it  using  grid diagrams representation for knots in lens  spaces and  knot Floer homology. Also applications outside mathematics seem to be promising (see \cite{BM, Ste}).

\paragraph{Main results of the paper and their contextualization}
In this paper we investigate the general problem of the equivalence  for link in lens spaces. As in the case of $\s3$, there are, at least, two possible notions of equivalence: under  ambient isotopy or under diffeomorphism of pairs.  In $\s3$ the two notions are easily related by the mirroring operation: that is, if two links are diffeo equivalent then either they are also isotopy equivalent or one is isotopic to the mirror image of the other. This is due to the fact that the diffeotopy group of $\s3$ is generated by an involution: in lens space the situation is more complicated since the diffeotopy group may be larger, and the action of the generators of the group on a diagram  of a link is not straightforward as in $\s3$.  Moreover,   all the combinatorial  representations of links in lens spaces deal with the isotopy equivalence and so  only  isotopy moves are described.  On the contrary, here we study the action of the elements of the diffeotopy group on three different kind of diagrams for  links:  disk \cite{CMM}, band \cite{Gb, HP} and grid \cite{BG} one.
We  provide a set of diffeo moves, that added to the already known isotopy ones, realize the combinatorial diffeo equivalence for links:  Figures \ref{t}, \ref{s+} and \ref{s-} describe them  for disk diagrams, Figure \ref{pt}, \ref{bandmove+} and ~\ref{bandmove-} are for band diagrams, while Figures \ref{g1},  \ref{g2} and  \ref{g3} are in the setting of grid diagrams.  The reason for considering different representations is that each of them, having its  own features,  has been used  fruitfully in literature to investigate different problems or  invariants. 



Our result, beside filling a gap, is useful when dealing with  issues in which the diffeo equivalence is the natural one. For example, in the study of     cosmetic couples,  (see \cite{BHW, Mat}),   link complement (see \cite{C, Ga}) and lift of the link  (see \cite{C, M1}). In this paper we investigate the lift of the link  that is, its  counterimage in $\s3$, under the universal cyclic covering map $\s3\to L(p,q)$. In \cite{BGH}  a grid diagram for the lift is constructed starting from a grid diagram of the link, while in \cite{M1}, the second author  uses disk diagrams to investigate whether the lift is a complete invariant or not. The general answer is negative both up to isotopy and diffeomorphism: the easiest example is given by the two axis of a lens space  $L(p,q)$ (i.e. the two core of the solid tori of the Heegaard decomposition): they both lift to the trivial knot but they are diffeo equivalent if and only if $q^2\equiv \pm 1\mod p$ (see \cite{Ma}). In \cite{M1} an example of a knot and a link with two components in $L(4,1)$ both lifting to the Hopf link is given. Moreover, using the cabling operation, this example is generalized to an  infinite family of couples of  links with the same number of components,  same homology class and same lifting. Even if the lift in general is not a complete invariant,   it may be classifying for some family of links.  A positive answer  in this direction is given  in \cite{C}, where it is shown that the lift is a complete invariant, up to isotopy, for the family of  knots in $L(2,1)\cong\mathbb{RP}^3$ lifting to a non trivial knot. In this paper, we deal with the same family, but in a general lens space: using results of \cite{BF,Sa} and the diffeo moves, we completely characterize, up to isotopy,  which knots among the family have the same lift,   in terms of the parameters of the lens space and of the amphicheirality of the lift (see Theorem	\ref{liftteo}).

\paragraph{Further development}

In \cite{Gb2}  a tabulations  of knots, represented via band diagrams and  up to isotopy equivalence is given. Diffeo-moves can be used to detect diffeo equivalent knots in the tabulation. For example, the two  knots $5_{26}$ and  $5_{27}$, which are non-isotopic in $L(p,q)$ with $p\leq 12$ and  conjecturally non-isotopic in the other lens spaces, are related by a  $\tau$ move. 

Another interesting development could be to study the diffeo moves in more general manifolds such Seifert one, for which isotopy equivalence moves are already known (see \cite{GM1}).

Regarding the lift there are two possible direction of investigation: the case in which the lift of the knot is a link with more then one component and the case in which the lift is the trivial knot. In \cite{C} the author
 deals with the second case for knots in the  projective space $L(2,1)$.

\paragraph{Paper organization}

In Section 2 we review the definition of lens spaces and of links in lens spaces, using disk,  band and grid diagrams, then we recall the isotopy equivalence of links in lens spaces and the set of Reidemeister-type moves that allows to understand when two disk,  band or diagrams represent isotopic links.

In Section 3 we provide new moves that state when two disk diagrams are equivalent up to diffeomorphism. These moves are also  translated into the setting of  band and grid diagrams.

Finally, in Section 4 we recall the lift construction and show that for primitive-homologous knots there are up to four different isotopy classes that have equivalent lift.

\end{section}

\section{Prerequisites and isotopy equivalence}
We will work in the category \emph{Diff} of differential manifold and differential maps. It is equivalent to the category \emph{PL} of piecewise linear manifolds and maps, and to the category \emph{Top}, since we do not consider wild links. As a consequence, we will  mix up  diffeomorphisms with homeomorphisms as well as diffeotopies  with isotopies. Moreover,  in figures we will often use PL representations. 

\paragraph{A lens model for lens spaces}\label{lens}

Fix two integer numbers, $p$ and $q$, such that $\gcd(p,q)=1$ and $ 0 \leqslant q < p$. 

Lens spaces may be defined through a lens model: in the  the $3$-dimensional ball $B^3$, let $E_{+}$ and $E_{-}$ be, respectively, the upper and the lower closed hemisphere of $\partial B^{3}$. The equatorial disk $B^{2}_{0}$ is defined by the intersection of the plane $x_{3}=0$ with $B^{3}$. Label with $N$ and $S$ respectively the points $(0,0,1)$ and $(0,0,-1)$ of $B^{3}$.
Let \mbox{$g_{p,q} \colon E_{+} \rightarrow E_{+}$} be the counterclockwise rotation of $2 \pi q /p$ radians around the $x_{3}$-axis, as represented in Figure~\ref{L(p,q)}, and let \hbox{$f_{3} \colon E_{+} \rightarrow E_{-}$} be the reflection with respect to the plane $x_{3}=0$.

\begin{figure}[htb!]                      
\begin{center}                         
\includegraphics[width=7cm]{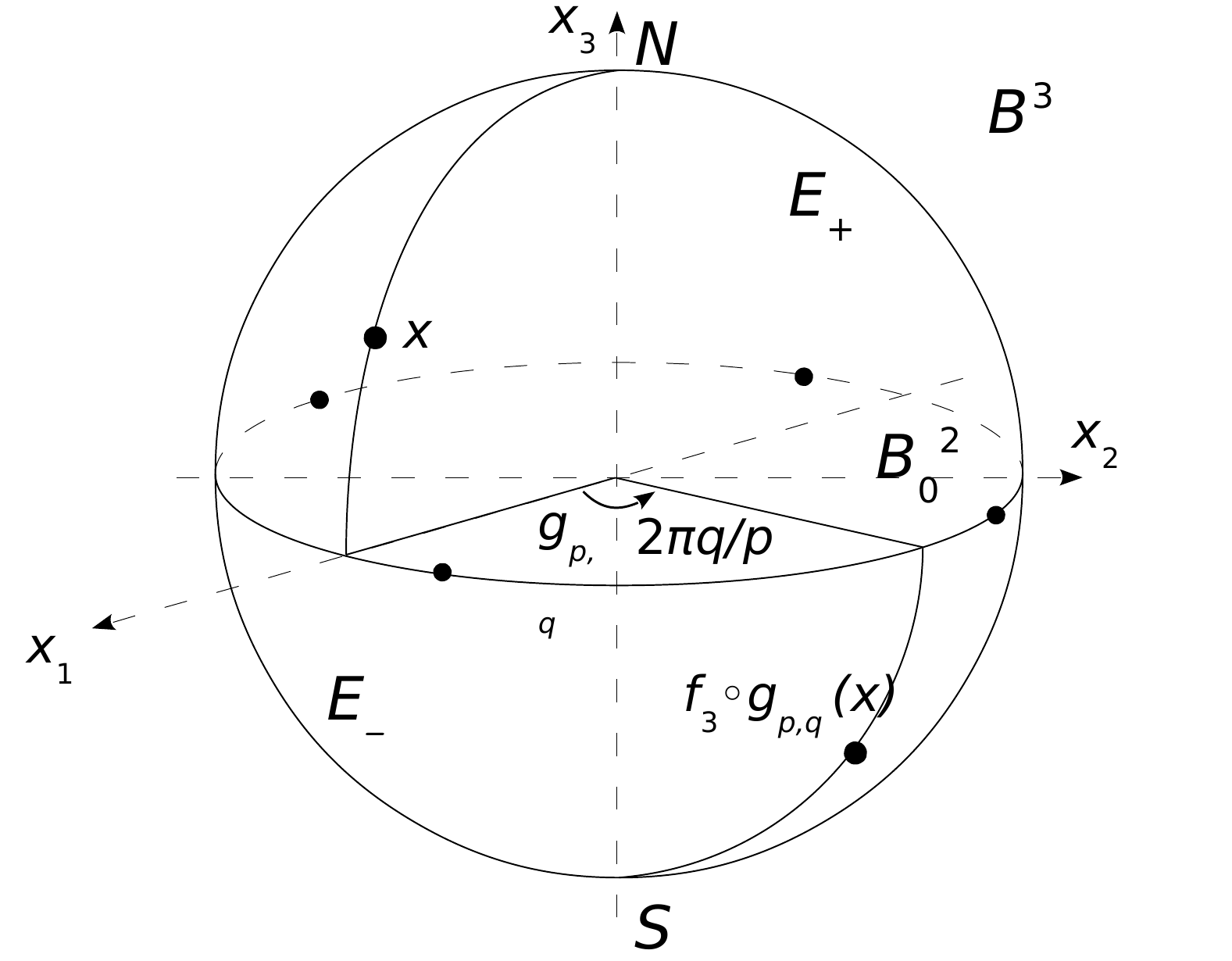}
\caption[legenda elenco figure]{Representation of $L(p,q)$ through lens model.}\label{L(p,q)}
\end{center}
\end{figure}

The \emph{lens space} $L(p,q)$ is the quotient of $B^{3}$ by the equivalence relation on $\partial B^{3}$ which identifies $x \in E_{+}$ with $f_{3} \circ g_{p,q} (x) \in E_{-}$. The quotient map is denoted by \mbox{$F \colon B^{3} \rightarrow B^{3} / \sim=L(p,q)$}. Note that, while generally on the boundary of the ball $B^{3}$ only two points are identified, on the equator \mbox{$\partial B^{2}_{0}=E_{+} \cap E_{-}$} each equivalence class contains $p$ points. We have $L(1,0)\cong \s{3}$ and  $L(2,1) \cong \rp{3}$.

\paragraph{Genus one Heegaard splitting model}\label{ls}

Heegaard splittings are one of the most powerful methods to represent closed orientable  $3$-manifolds (see \cite{Heg}). We can define   a genus $g$ Heegaard spitting of a given manifold either as a closed orientable genus $g$ surface embedded in the manifold  whose complement consists of two genus $g$ handlebodies, or  as an identification between the boundaries of two copies of the a genus $g$ handlebody through (the isotopy class of)  an orientation reversing  diffeomorphism giving as result the considered manifold. Lens spaces and $\s2\times \s1$ are those $3$-manifolds admitting a genus one Heegaard splittings: i.e., if  $V_{1}$ and $V_{2}$ are two copies of a solid torus $V=\s1 \times B^{2} \subset \C \times \C$, a genus one Heegaard splitting $V_{1}\cup_{\varphi_{p,q}}V_{2}$ of the lens space $L(p,q)$ is the gluing of the two solid tori $V_{1}$ and $V_{2}$ via the diffeomorphism of their boundaries $\varphi_{p,q} \colon \partial V_{2} \rightarrow \partial V_{1}$ that sends the  curve $\beta=\{\ast\} \times \partial B^{2}$ to the curve $q \beta +p \alpha$, where $\alpha=\s1\times \{\ast\}$, with $\ast$ a fixed point.
In Figure~\ref{lenshsx} it is illustrated the case $L(5,2)$  while Figure~\ref{QtoHS} explains, in the case of $L(5,2)$,  how to get the Heegaard splitting model starting from the lens one, and vice versa.

\begin{figure}[htb!]                      
\hspace{-10pt}
\begin{center}                         
\includegraphics[width=10cm]{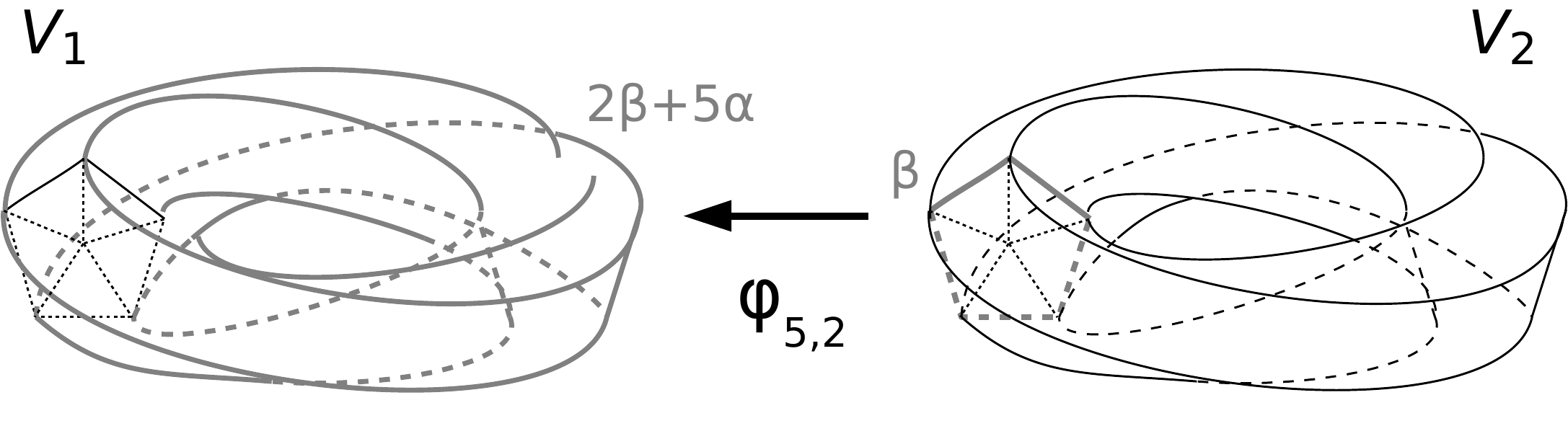}
\caption[legenda elenco figure]{Heegaard splitting of $L(5,2)$.}\label{lenshsx}
\end{center}
\end{figure}

\begin{figure}[h!]                      
\begin{center}                         
\includegraphics[width=12.2cm]{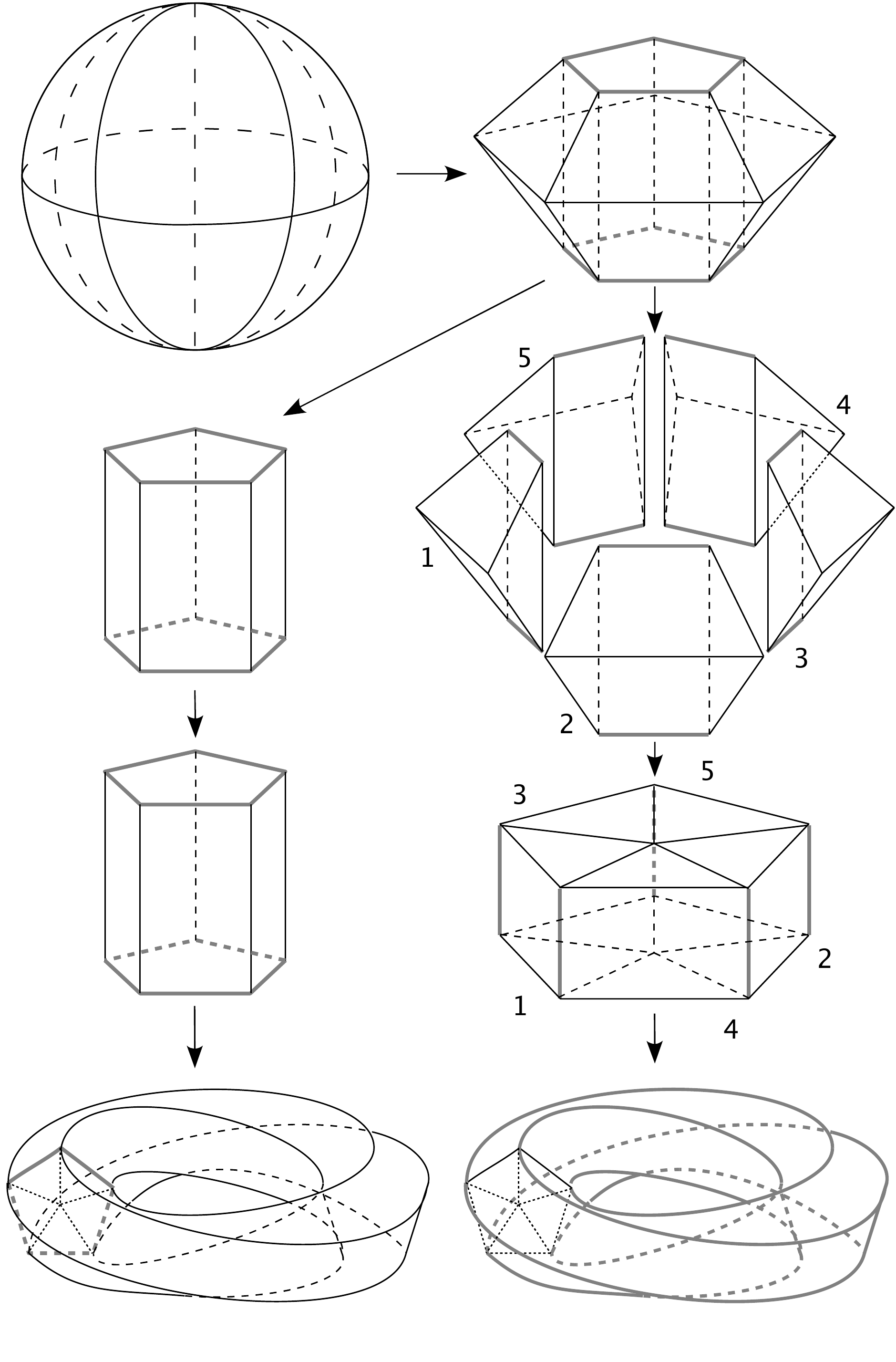}
\caption[legenda elenco figure]{From the lens model to Heegaard splitting one of $L(5,2)$.}\label{QtoHS}
\end{center}
\end{figure}

\paragraph{Links in lens spaces and isotopy equivalence} A link $L$ in a $3$-manifold $M$ is a pair $(M,L)$, where $L$ is a submanifold of $M$ diffeomorphic to the disjoint union of $\nu$ copies of $\s1$, with $\nu > 0$. A \emph{component} of $L$ is each connected component of the topological space $L$. When $\nu=1$ the link is called a \emph{knot}.
We usually refer to $L \subset M$ meaning the pair $(M,L)$.
A link $L \subset M$ is \emph{trivial} if its components bound embedded pairwise disjoint $2$-disks $B^{2}_{1}, \ldots, B^{2}_{\nu}$ embedded in $M$.   Links may be oriented or not, according to the need.

There are at least two possible definitions of link equivalence. At first, we focus on the equivalence by ambient isotopy: two links $L,L'\subset M$ are called \emph{isotopy equivalent} if there exists a continuous map $H\colon M \times [0,1] \rightarrow M$ where, if we define $h_{t}(x):=H(x,t)$, then $h_{0}=id_{M}$, $h_{1}(L)=L'$ and $h_{t}$ is a diffeomorphism of $M$ for each $t \in [0,1]$.


\paragraph{Disk diagram}\label{disk}

Since we are not interested in the case of $\s3$, we assume $p>1$. Intuitively, a \emph{disk diagram} of a link $L$ in $L(p,q)$, represented by the lens model, is a regular projection of $L'=F^{-1}(L)$ onto the equatorial disk of $B^{3}$, with the resolution of double points with overpasses and underpasses. 
In order to have a more comprehensible diagram, we label with $+1, \ldots, +t$ the endpoints of the projection of the link coming from the upper hemisphere, and with $-1, \ldots, -t$ the endpoints coming from the lower hemisphere, respecting the rule $+i \sim -i$. 
An example is shown in Figure~\ref{link3}. The rigorous definition can be found in \cite{CMM}.

\begin{figure}[h!]                      
\begin{center}                         
\includegraphics[width=8cm]{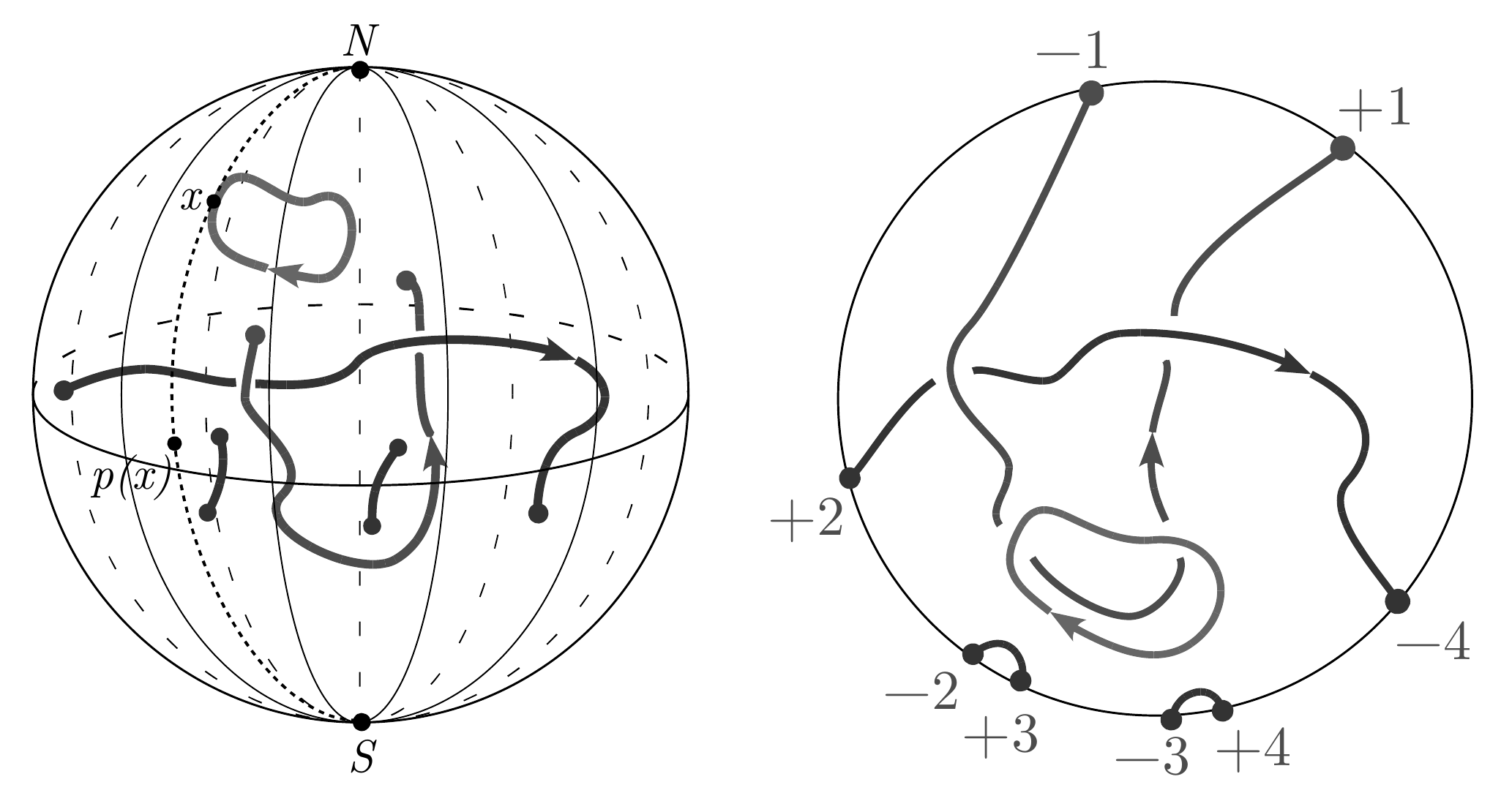}
\caption[legenda elenco figure]{A link in $L(9,1)$ and its corresponding disk diagram.}\label{link3}
\end{center}
\end{figure}

\paragraph{Generalized Reidemeister moves for disk diagrams}

In \cite{CMM}, a Reidemeister-type theorem is proved for disk diagrams of links in lens spaces.
The \emph{generalized Reidemeister moves} on a disk diagram of a link $L \subset L(p,q)$, are the moves $R_{1}, R_{2}, R_{3}, R_{4}, R_{5}, R_{6}$ and $R_{7}$ of Figure~\ref{vR1-R7}. 

\begin{figure}[tbh!]                      
\begin{center}                         
\includegraphics[width=11cm]{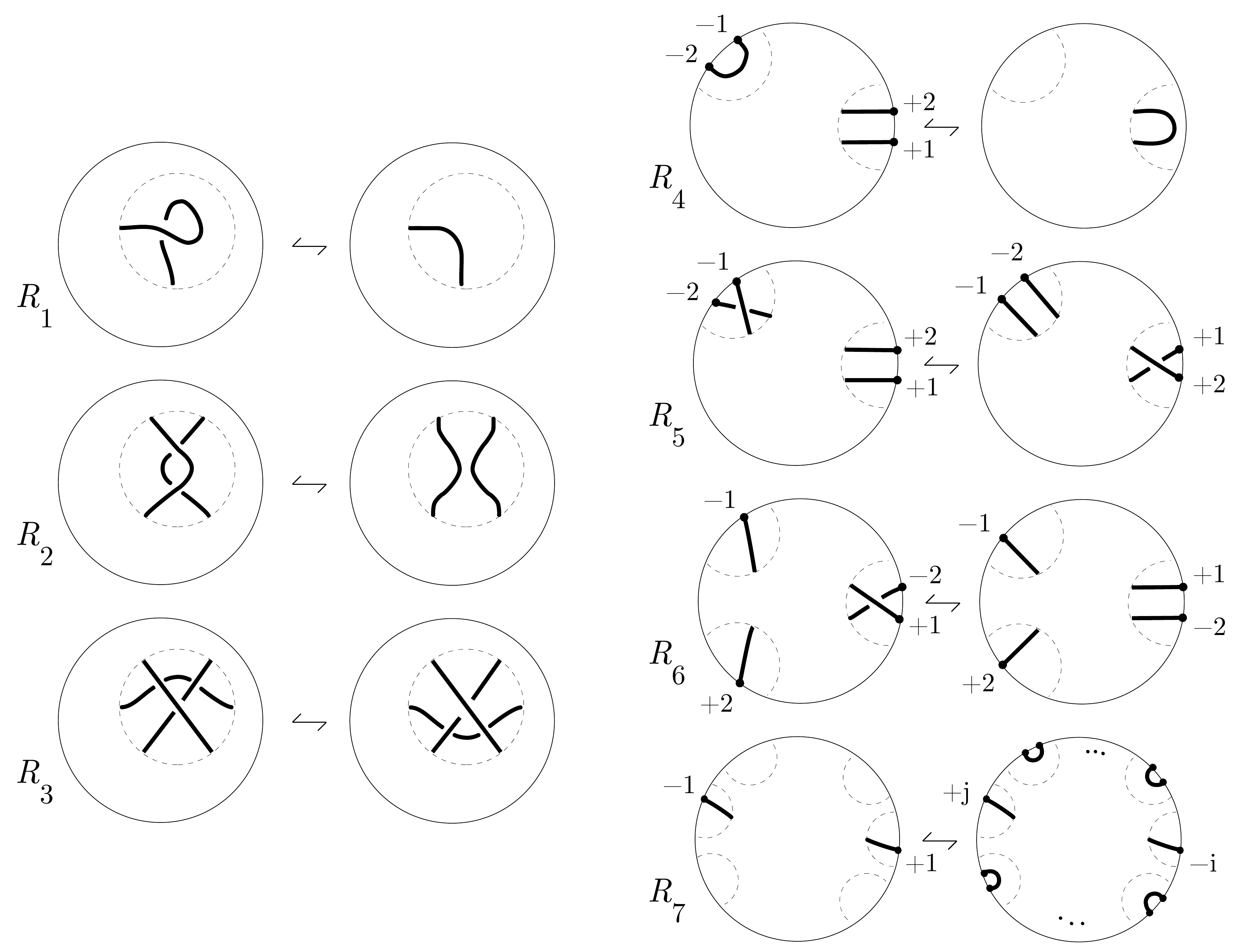}
\caption[legenda elenco figure]{Generalized Reidemeister moves.}\label{vR1-R7}
\end{center}
\end{figure}

Observe that, when $p=2$ (i.e., the ambient space is $\mathbb{RP}^3$),  the moves coincide with those described in \cite{Dr}, since $R_{5}$ and $R_{6}$ are equal, and $R_{7}$ is a trivial move.


A disk diagram is called to be \emph{standard} if the labels on its boundary points, read according to a fixed  orientation on $\partial B^{2}_{0}$, are $(+1, \ldots, +t, -1, \ldots, -t)$.  Every disk diagram can be reduced to a standard disk diagram using generalized Reidemeister moves (see \cite{M1}).

\paragraph{Band diagram}\label{band}

Another possible representation of links in lens spaces is band diagrams. Represent  $L(p,q)$ via genus one Heegaard splitting where we can think that each  solid torus $V_1,V_2$ is   embedded in $\mathbb R^3$ in a standard way (i.e., as  depicted in Figure \ref{lenshsx}). By general position theory, we can suppose that a link $L \subset L(p,q)$ is contained  inside one of the two solid tori, and thus we can regularly project $L$ onto the plane   containing  $\s1 \times \{0\}$ and add a  dot being the projection of the axis of rotational symmetry of the torus, as depicted in the left part of Figure \ref{banddiagram}. Such a representation is called \textit{punctured disk diagram}: in order to obtain a \textit{band diagram}  it is enough to cut open  the  punctured disk diagram with a ray starting from the dot, as depicted in the right part of  Figure \ref{banddiagram}. A rigorous definition can be found in \cite{GM}.

\begin{figure}[h!]                      
\centering    
\includegraphics[width=8cm]{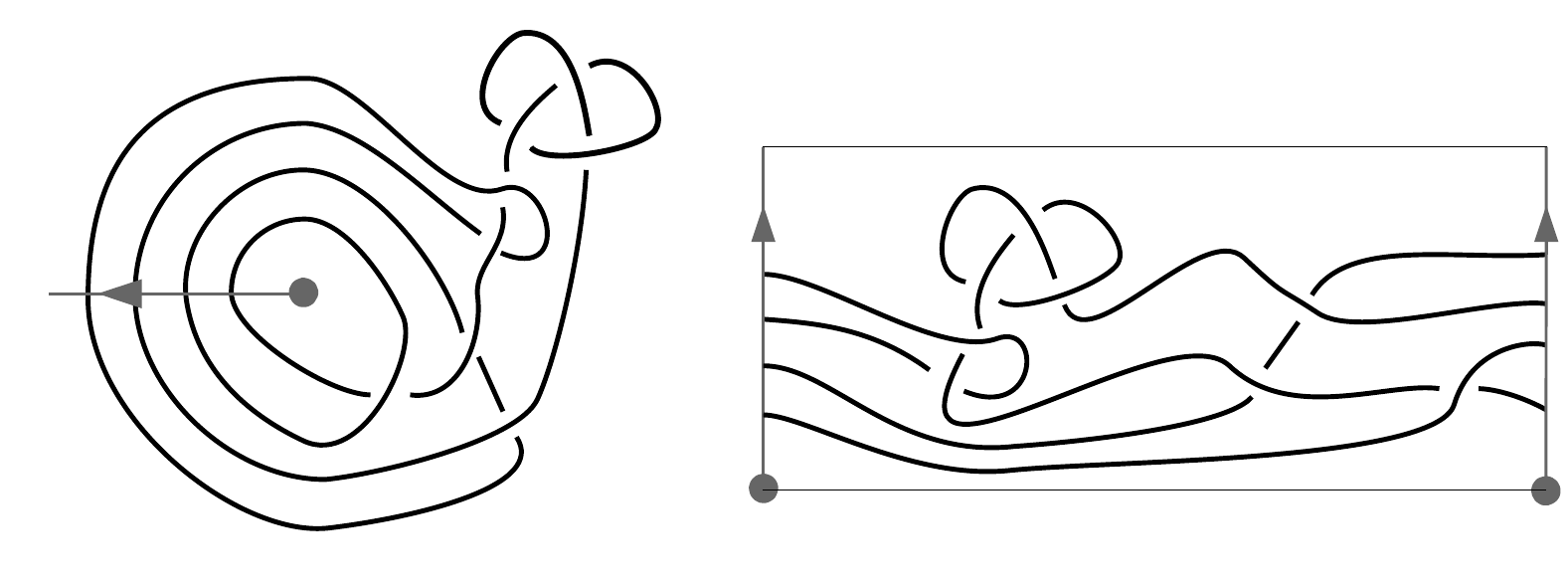}
\caption[legenda elenco figure]{Example of punctured disk and band diagram.}\label{banddiagram}
\end{figure}


\paragraph{Generalized Reidemeister moves for band diagrams}

For band diagrams the isotopy equivalence is established by the classical three Reidemeister moves together with a non local move, called  \textit{slide move} and depicted in Figure \ref{SLmove}, depending on the parameters $p,q$ and corresponding to an isotopy in which an arc of the link passes through the dot. 

\begin{figure}[h!]                      
\centering    
\includegraphics[width=9.5cm]{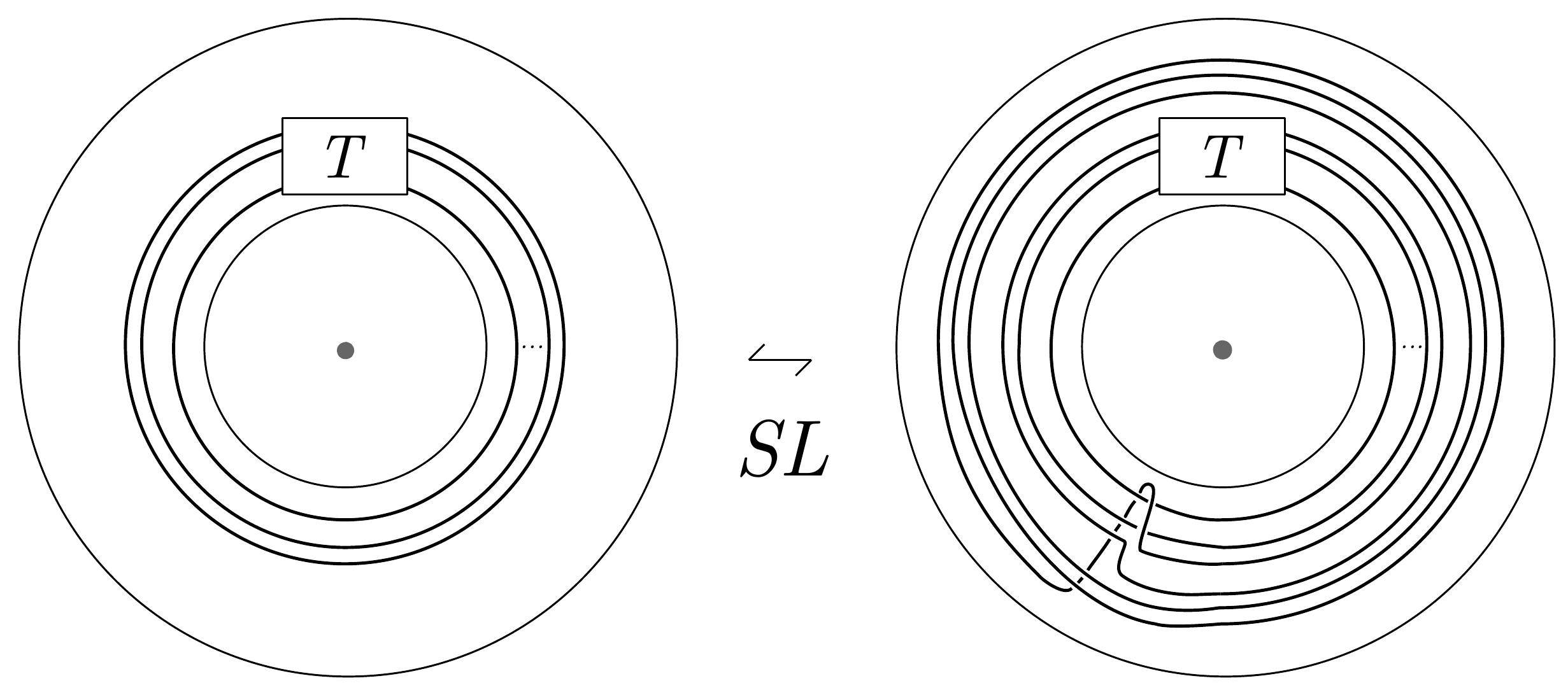}
\caption[legenda elenco figure]{Slide move for punctured disk/band diagrams}\label{SLmove}
\end{figure}

\paragraph{Transformation between disk and band diagrams}

The complete description of how to transform a band diagram into a disk diagram and vice versa can be found in \cite{GM}. Denote with
  $\Delta_{t}$ the element  $(\sigma_{1}\sigma_{2} \cdots \sigma_{t-1})( \sigma_{1}\sigma_{2} \cdots \sigma_{t-2}) \cdots (\sigma_{1})$ of the  braid group, where $\sigma_i$ are the standard Artin generators,    represented in Figure \ref{delta}  that corresponds to a half twist of all the strands. In Figures \ref{LB} and \ref{BL} are described  the  transformations between disc and band diagrams.

\begin{figure}[h!]                      
\centering    
\includegraphics[width=8cm]{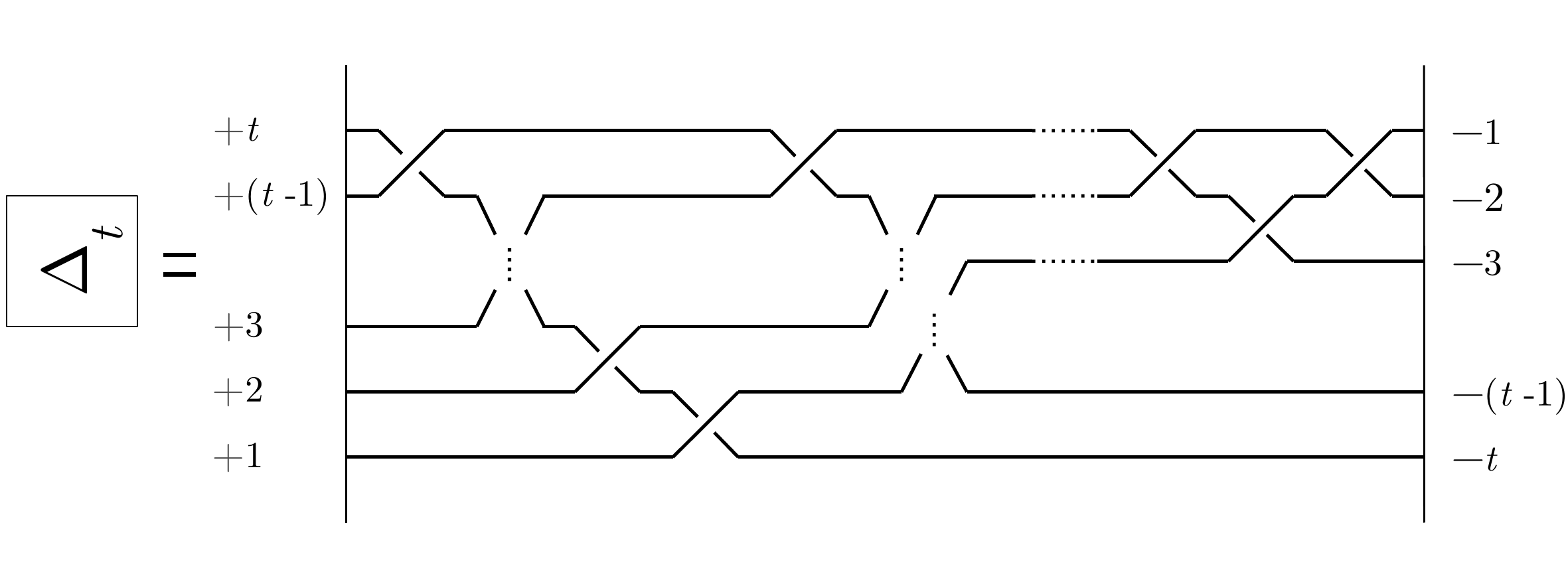}
\caption[legenda elenco figure]{The braid group element $\Delta_{t}$.}\label{delta}
\end{figure}

\begin{figure}[h!]                      
\centering    
\includegraphics[width=9cm]{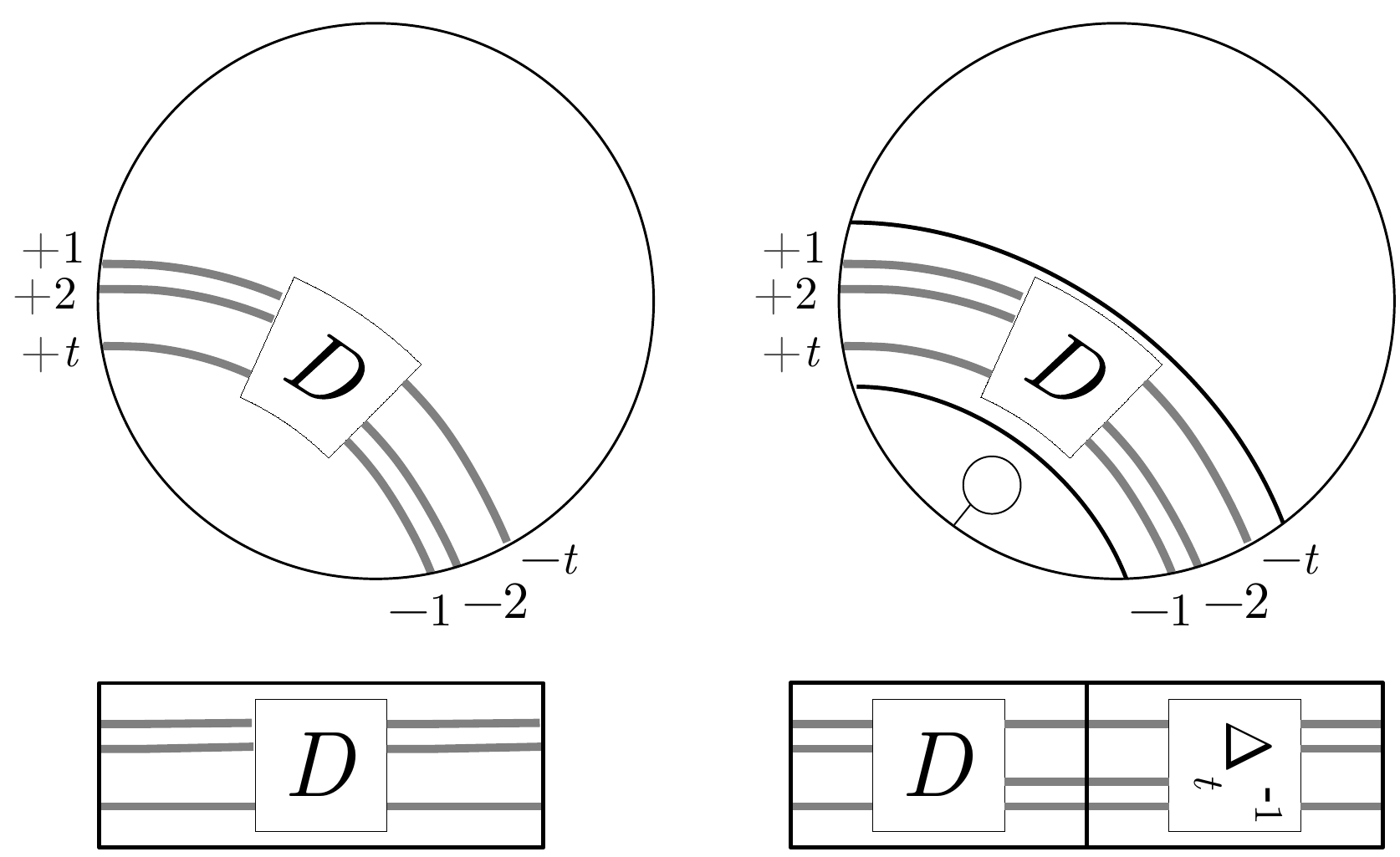}
\caption[legenda elenco figure]{Transformation from a disk diagram to a band diagram.}\label{LB}
\end{figure}

\begin{figure}[h!]                      
\centering    
\includegraphics[width=9cm]{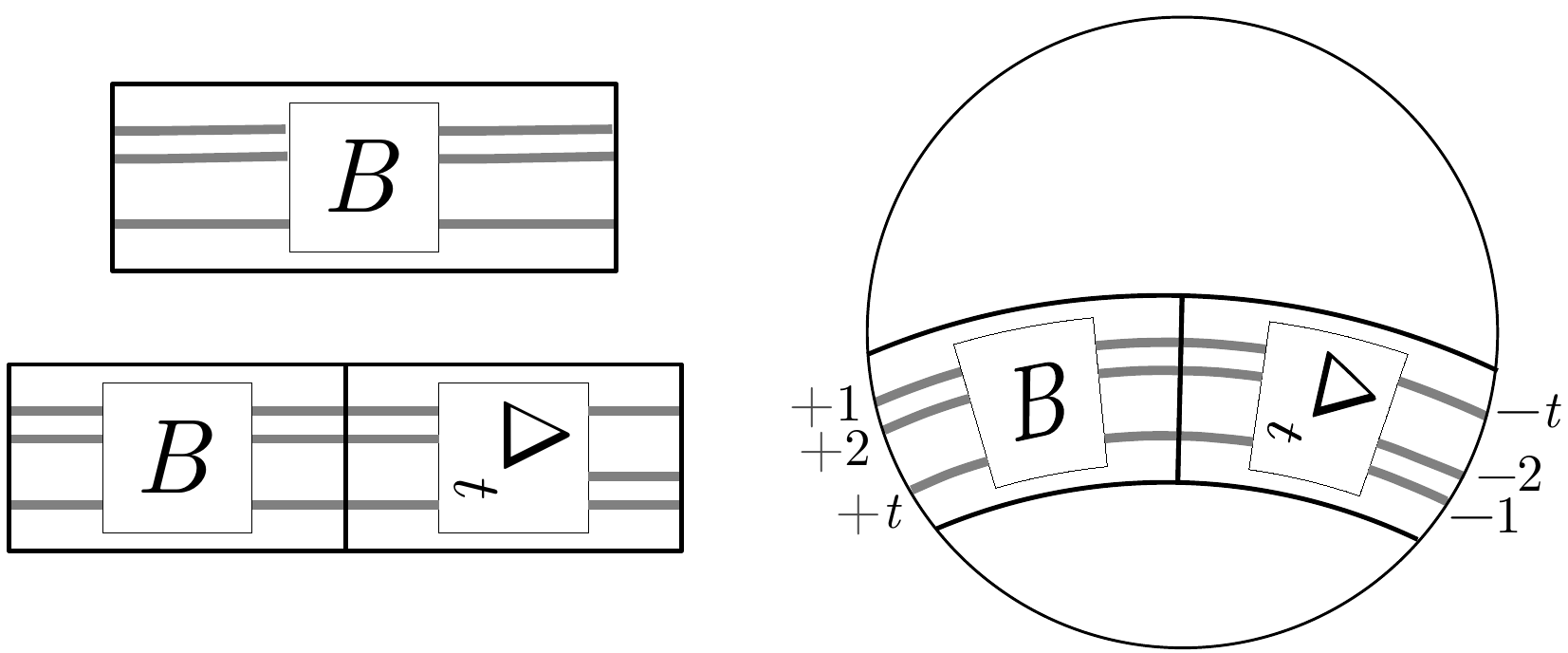}
\caption[legenda elenco figure]{Transformation from a band diagram to a disk diagram.}\label{BL}
\end{figure}

\paragraph{Grid diagrams}
\label{grid}
Another type of representation for links in lens spaces relying on genus one Heegaard splittings is the \textit{grid diagram} one developed in   \cite{BG}.  Using general position theory and PL-approximation, we can suppose that  a link in $L(p,q)$ intersects transversely the Heegaard torus in $2n$ points  and that the projection of the link over the Heegaard torus is ``parallel'' with respect to  a coordinate system of the torus consisting of  the boundaries of a fixed meridian disk in each solid tori. As a consequence, referring to Figure \ref{grid},  each link in $L(p,q)$ can be represented by  marking the intersections of the link with the Heegaard torus represented as  a rectangular grid  diagram  obtained by juxtaposing $p$ boxes $n\times n$.  Clearly the vertical sides of the rectangle as well as  the  horizontal ones are identified  together   according to the arrows and  the upper horizontal side is glued with the lower one after performing  a left shift of $q$ boxes, as suggested by the indices;  the horizontal curves (resp. the vertical ones) divides the torus into $n$ horizontal (resp. vertical) annuli called  \textit{rows} (resp. \textit{columns}). In each row (resp. column) there is one and only one   $X$ and one and only one $O$: according to a fixed orientation of the link and   of the  normal bundle of the torus, the difference between the two types of marking depends whether,    the tangent vector to the link in the intersection point  is coherent or not to the normal bundle of the torus. The  number $n$ is called \textit{grid number} of the diagram.

 \begin{figure}[htb!]                      
\begin{center}                         
\includegraphics[width=13cm]{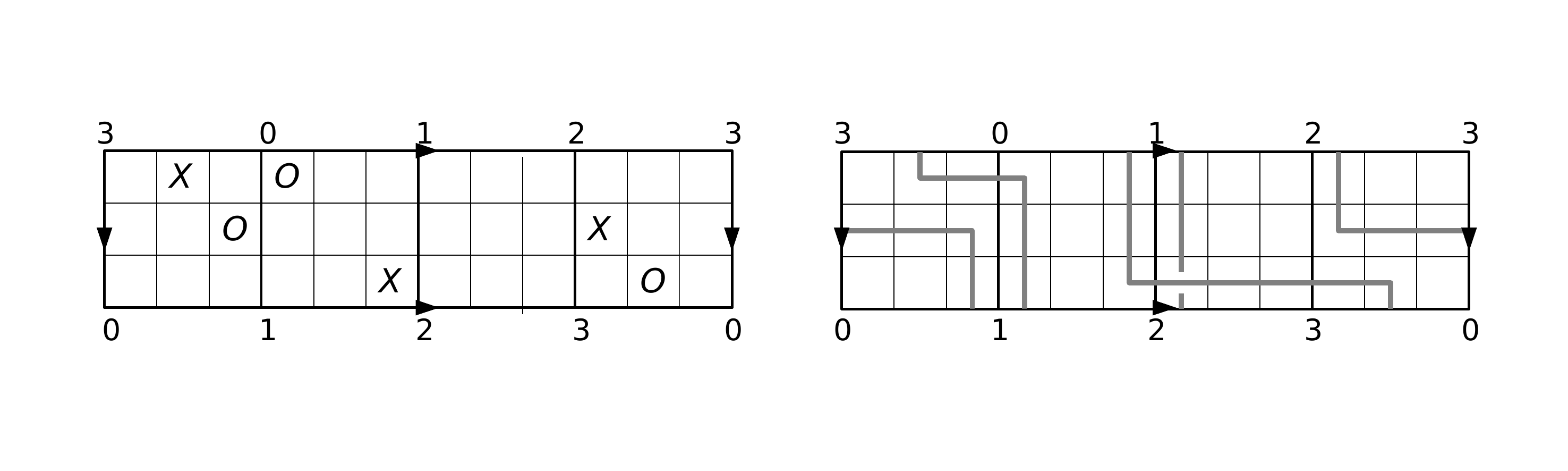}
\caption[legenda elenco figure]{A grid diagram with grid number $3$ in $L(4,1)$ and the reconstruction of the represented link.}\label{grid}
\end{center}
\end{figure}

\paragraph{Generalized Reidemeister moves for grid diagrams}
For grid diagram  the isotopy equivalence is established (see \cite{BG}) via  two kinds of moves: \textit{(de)stabilization} depicted  in Figure \ref{stab},  increasing or decreasing the grid number by one,  and   \textit{non-interleaving commutation} depicted   in Figure  \ref{comm} and exchanging two adjacent columns or rows. The adjective non-interleaving  refers to the following fact:  the annulus  $A$ 
consisting of  the two exchanged  columns (or rows)    is divided into $pn$ bands  by the rows (columns); if  $s_1$ and $s_2$ are the two bands  containing the markings of one of the exchanged column (row),  the markings of the other column (row) must lie in a  different component of $A-s_1-s_2$.

\begin{figure}[h!]                      
\begin{center}                         
\includegraphics[width=12cm]{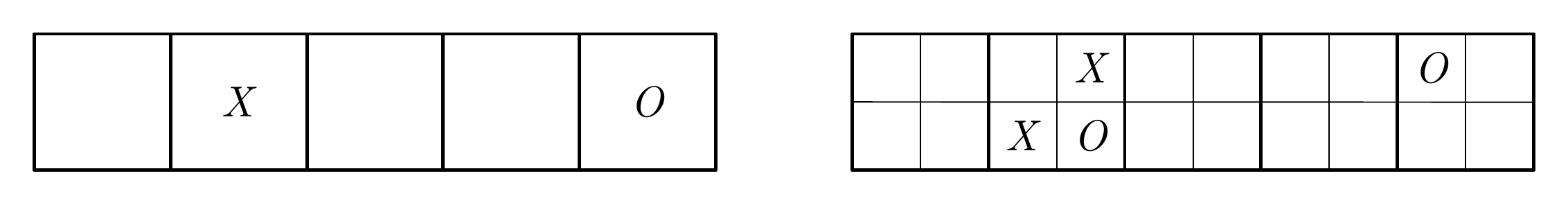}
\caption[legenda elenco figure]{An example of (de)stabilization.}\label{stab}
\end{center}
\end{figure}

\begin{figure}[h!]                      
\begin{center}                         
\includegraphics[width=12cm]{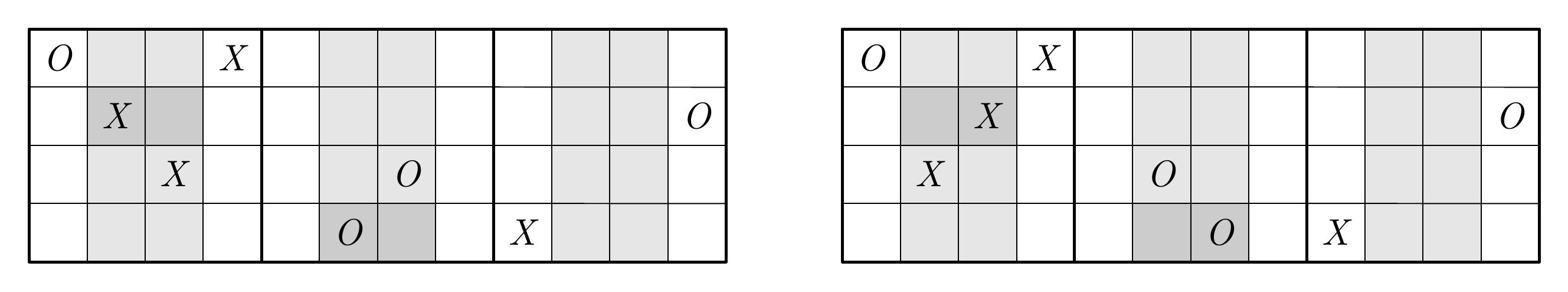}
\caption[legenda elenco figure]{An example of non-interleaving commutation.}\label{comm}
\end{center}
\end{figure}

\paragraph{Transformation between disk and grid diagram} 
The transition between grid and disk diagram is sketched in Figures \ref{GL} and \ref{LG} and  described in \cite{CMR}.
 
\begin{figure}[h!]                      
\begin{center}                         
\includegraphics[width=10cm]{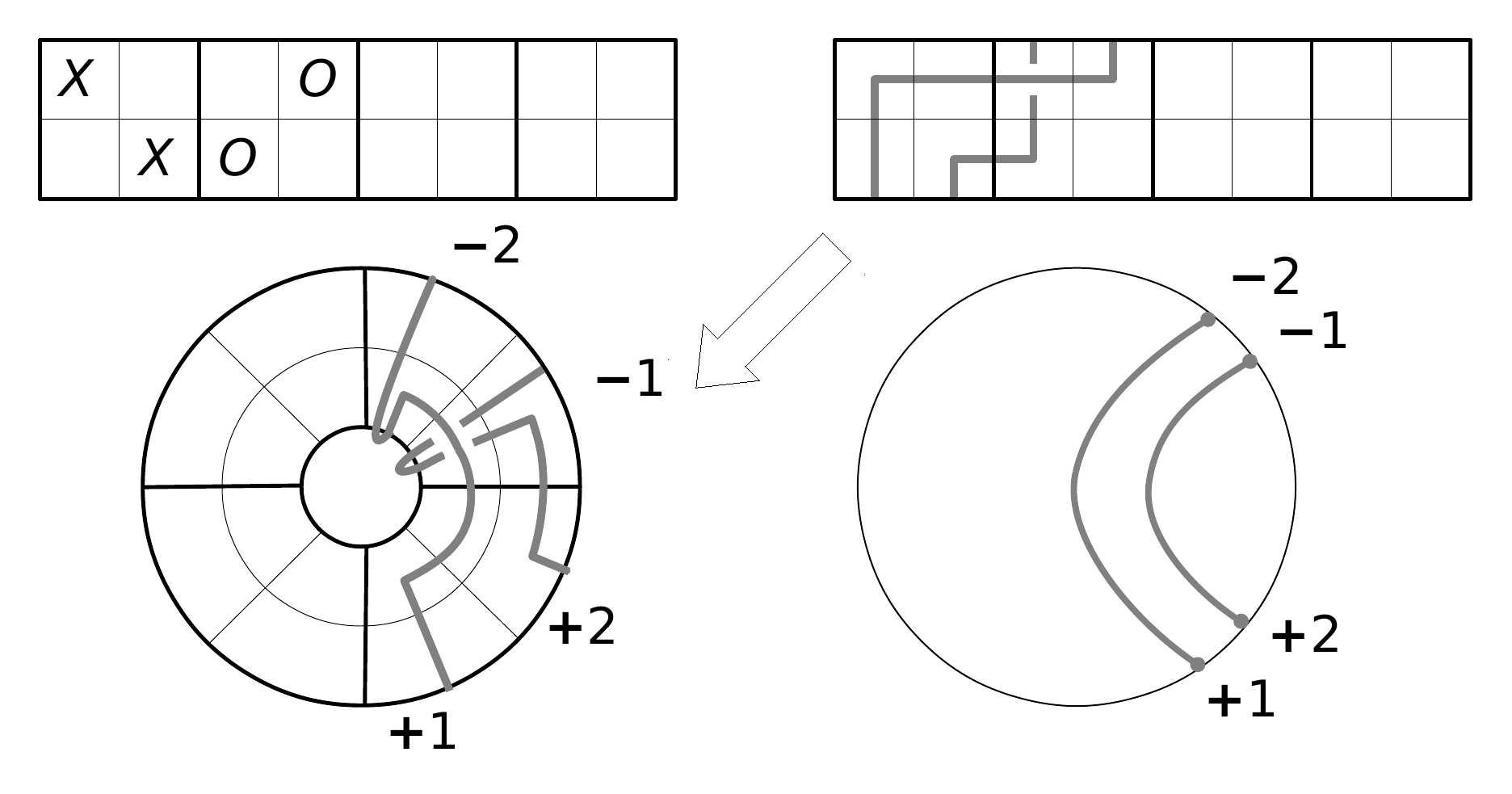}
\caption[legenda elenco figure]{From a grid diagram to a disk diagram  in $L(4,1)$.}\label{GL}
\end{center}
\end{figure}

\begin{figure}[h!]                      
\begin{center}                         
\includegraphics[width=10cm]{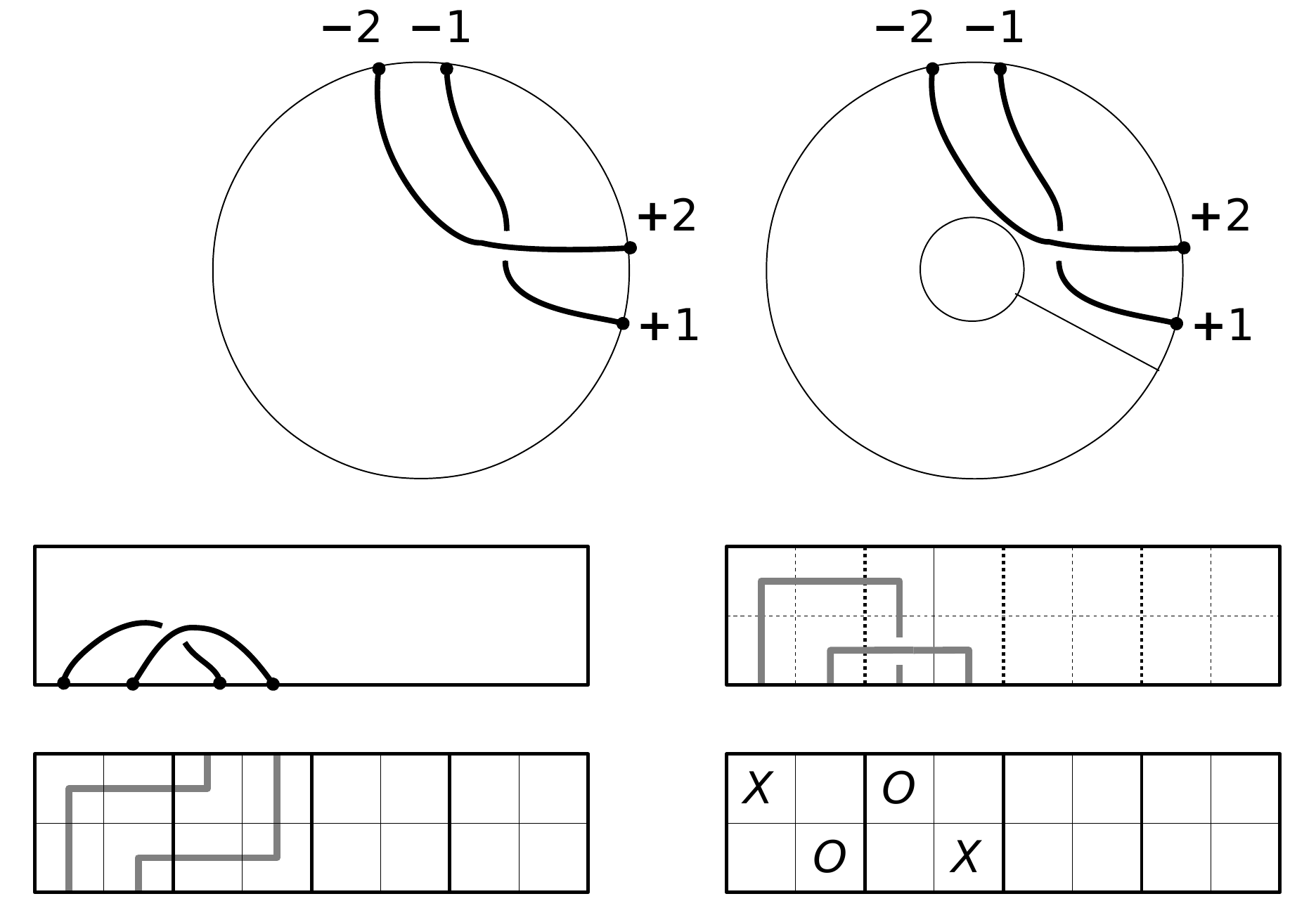}
\caption[legenda elenco figure]{From a disk diagram  to a grid diagram in $L(4,1)$.}\label{LG}
\end{center}
\end{figure}

\paragraph{The homology class of knots}
The \textit{homology class} of an oriented knot $K \subset L(p,q)$ is $[K] \subset H_{1}(L(p,q)) \cong \mathbb{Z}_{p}$. If $K$ is represented via disk diagram the homology class $[K]$ is the algebraic intersection number,  obviously mod $p$,  between the knot $K$ and the oriented boundary of the disk diagram, while if it is represented via  band  (resp. grid diagram), it is the algebraic intersection number between $K$ and one of the vertical (resp.  horizontal) oriented sides  of the rectangle. 
Notice that the isotopy moves cannot change the homology class of the knots.


\section{Diffeo-equivalence}

\paragraph{Diffeo-equivalence definition}

A weaker equivalence relation for the set of links in a lens space is the one by diffeomorphism of pairs.   
Two links $L$ and $L'$ in $M$ are \emph{diffeo equivalent} if there exists a diffeomorphism of pairs $h: (M,L) \rightarrow(M,L')$, that is to say a diffeomorphism $h:M \rightarrow M$ such that $h(L)=L'$. It is not necessary that the diffeomorphism is orientation preserving.

\begin{oss}
Two isotopy equivalent links $L$ and $L'$ in $M$ are necessarily diffeo equivalent, since from the ambient isotopy $H\colon M \times [0,1] \rightarrow M$, the map $h_{1}\colon (M,L) \rightarrow (M,L')$ is a diffeomorphism of pairs.
\end{oss}

The two definitions are equal for links in $\s3$ up to an orientation reversing diffeomorphism. For the lens spaces, this is not true, as we will see in the following.

\paragraph{The Diffeotopy group of lens spaces}

In order to describe the group of the diffeomorphism of the lens space up to isotopy, we have to consider the three following elements, which preserve the Heegaard torus $\partial V_{1}=\varphi_{p,q}(\partial V_{2})$.
Let $\tau$ be the diffeomorphism of $L(p,q)$, preserving $V_{1}$ and $V_{2}$, such that it is an involution and can be expressed into coordinates by $\tau(u,v)=(\bar{u},\bar{v})$ for each solid torus. Let $\sigma_{+}$ be the diffeomorphism exchanging $V_{1}$ and $V_{2}$, parametrized by $\sigma_{+}: (u,v) \in V_{1} \longleftrightarrow (u,v) \in V_{2}$. It is also an involution and it commutes with $\tau$.
At last, the diffeomorphism $\sigma_{-}$ exchanges $V_{1}$ and $V_{2}$ by $\sigma_{-}: (u,v) \in V_{1} \longleftrightarrow (\bar{u},v) \in V_{2}$. It is of order $4$ and $\sigma_{-}^{2}=\tau$. 

\begin{teo}[\cite{Bo}, \cite{HR}]\label{diffeotopy}
The group of diffeomorphism of the lens space $L(p,q)$, up to isotopy, is isomorphic to:
\begin{itemize}\itemsep-3pt
\item $\z2$ if $p=2$, it is generated by $\sigma_{-}$; 
\item $\z2$ if  $q \equiv \pm 1 \mod p$ and $p \neq 2$, it is generated by $\tau$;
\item $\z2 \oplus \z2$ if  $q^2 \equiv + 1 \mod p$ and $q \not \equiv \pm 1 \mod p$, it is generated by $\tau$ and $\sigma_{+}$;
\item $\z4$ if $q^2 \equiv - 1 \mod p$ and $p \neq 2$, it is generated by $\sigma_{-}$;
\item $\z2$ if $q^{2} \not \equiv \pm 1 \mod p$, it is generated by $\tau$.
\end{itemize}

\end{teo}

\paragraph{Moves for diffeomorphisms of lens spaces for disk diagrams}

\begin{teo}\label{dm}
Let $D_{L}$ be a standard disk diagram of the link $L$ in the lens space $L(p,q)$.
Then the effects of $\tau$, $\sigma_{+}$ and $\sigma_{-}$ on $D_{L}$ are the ones described, respectively, in Figures \ref{t}, \ref{s+} and \ref{s-}. 
The integer $h$, necessary for the description of the moves corresponding to $\sigma_{+}$ and $\sigma_{-}$, is defined  as follows:
\begin{itemize}\itemsep-3pt
\item  if $q^2 \equiv + 1 \mod p$ and $q \not \equiv \pm 1 \mod p$, then $h=(q^2 -1)/p$;
\item  if $q^2 \equiv - 1 \mod p$, then $h=(q^2 +1)/p$;
\end{itemize}
where $L$  denotes a $t$-tangle,  $L$ mirrored  denotes the mirror of the tangle $L$  and  $s(L)$  denotes the  tangle  $L$ with all the crossings exchanged.

\begin{figure}[h!]      
\center                
\includegraphics[width=10cm]{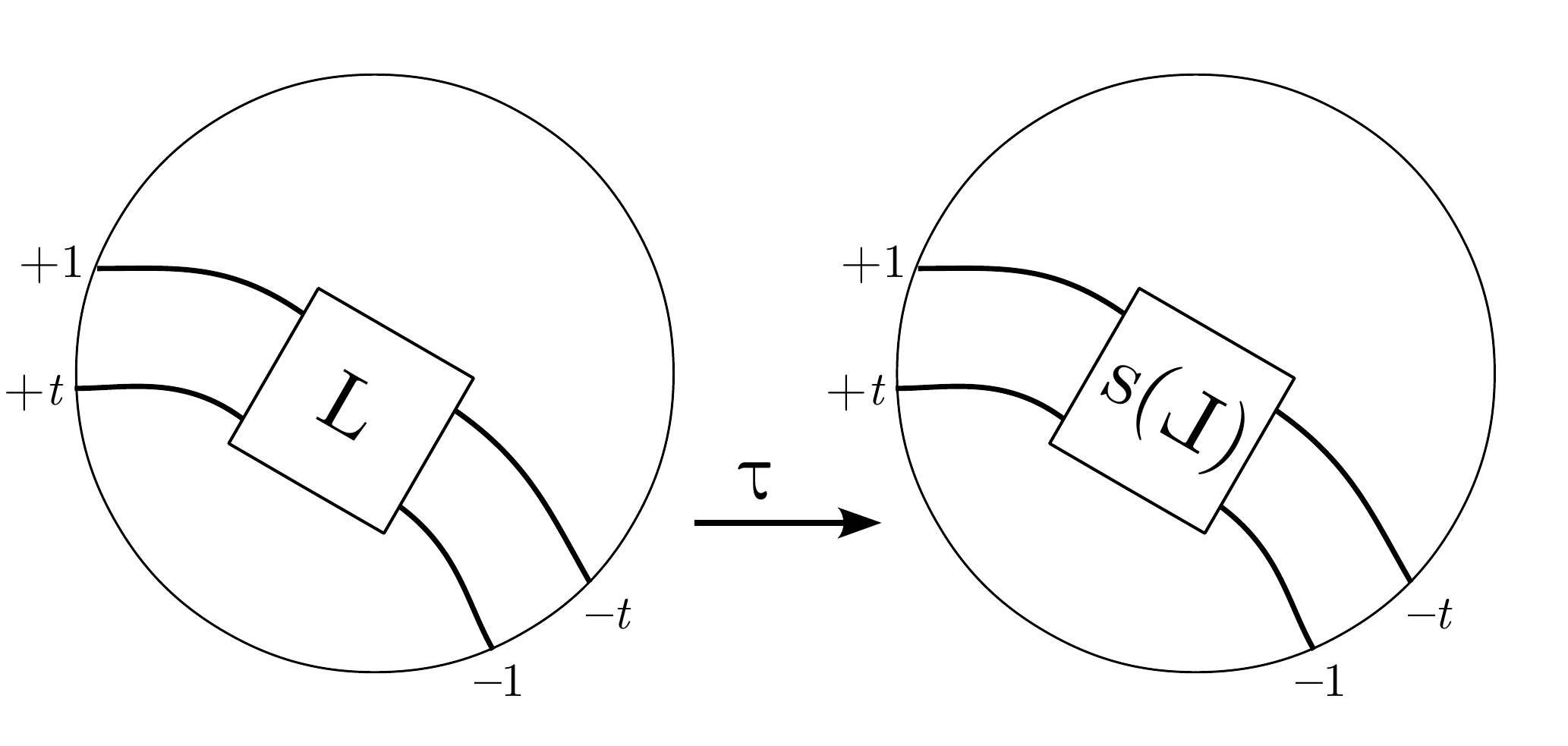}
\caption[legenda elenco figure]{Effect of $\tau$ on a standard disk diagram in $L(p,q)$.}\label{t}
\end{figure}
\begin{figure}[h!]  
\center                   
\hspace{-5mm}
\includegraphics[width=13cm]{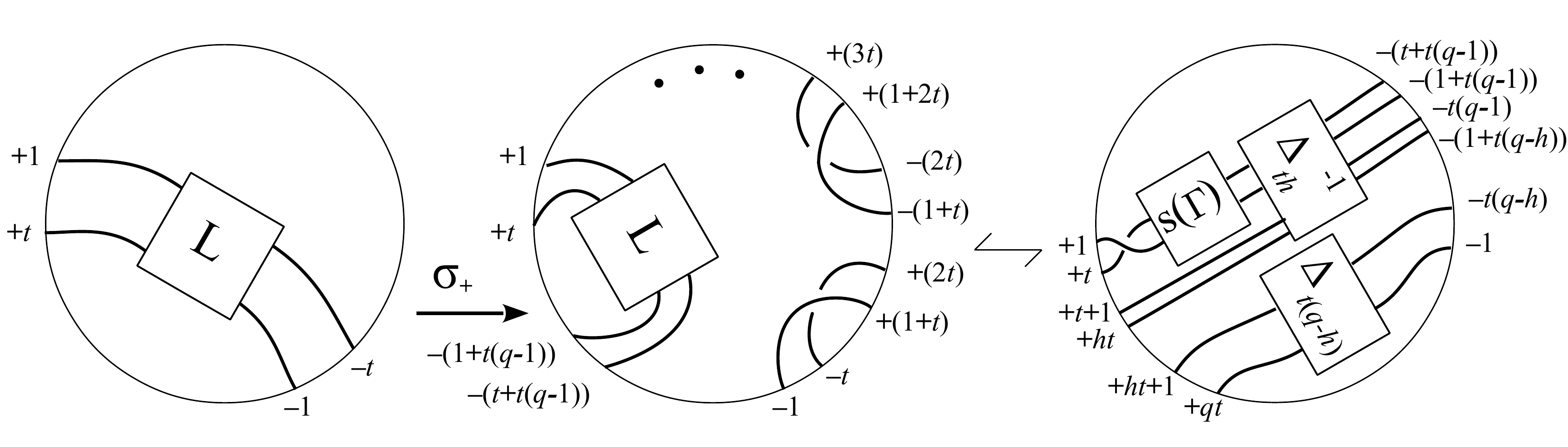}
\caption[legenda elenco figure]{Effect of $\sigma_{+}$ on a standard disk diagram in $L(p,q)$.}\label{s+}
\end{figure}
\begin{figure}[h!]          
\center            
\hspace{-5mm}
\includegraphics[width=13cm]{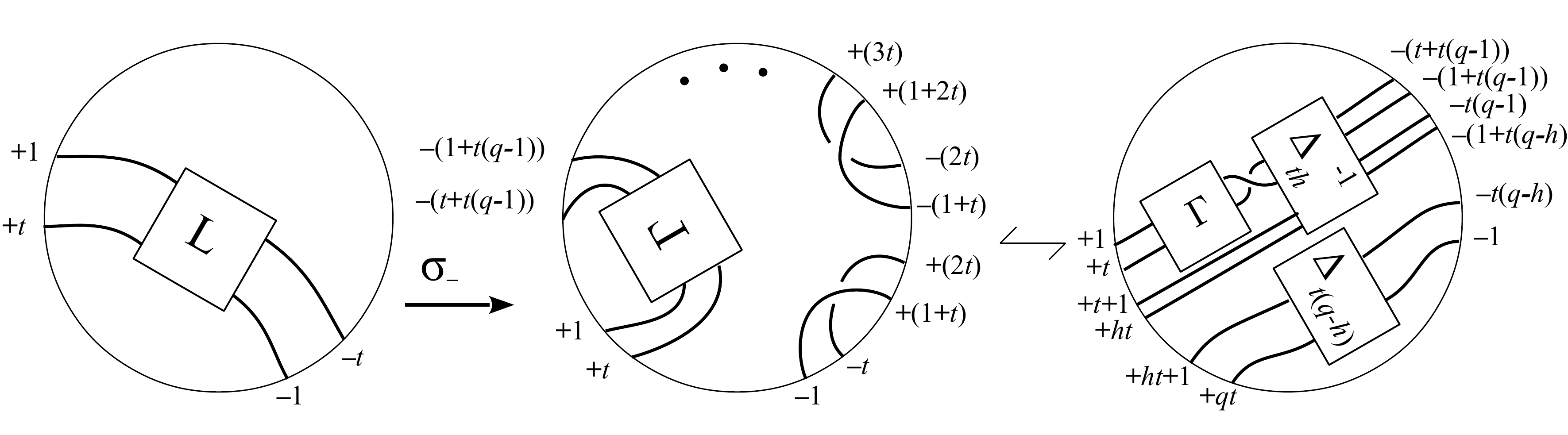}
\caption[legenda elenco figure]{Effect of $\sigma_{-}$ on a standard disk diagram in $L(p,q)$.}\label{s-}
\end{figure}
\end{teo}
\begin{proof}
For the $\tau$ move, it is easier to see the effect of $\tau$ on a band diagram, as displayed in Figure~\ref{pt}, and then recover the disk diagram as shown in Section~2.
Since the band diagram is a diagram in one of the two solid tori giving the lens space, the diffeomorphism $\tau(u,v)=(\bar{u},\bar{v})$ corresponds to a mirror on the band diagram plus the exchange of all the crossings.
\begin{figure}[h!]      
\center                
\includegraphics[width=12cm]{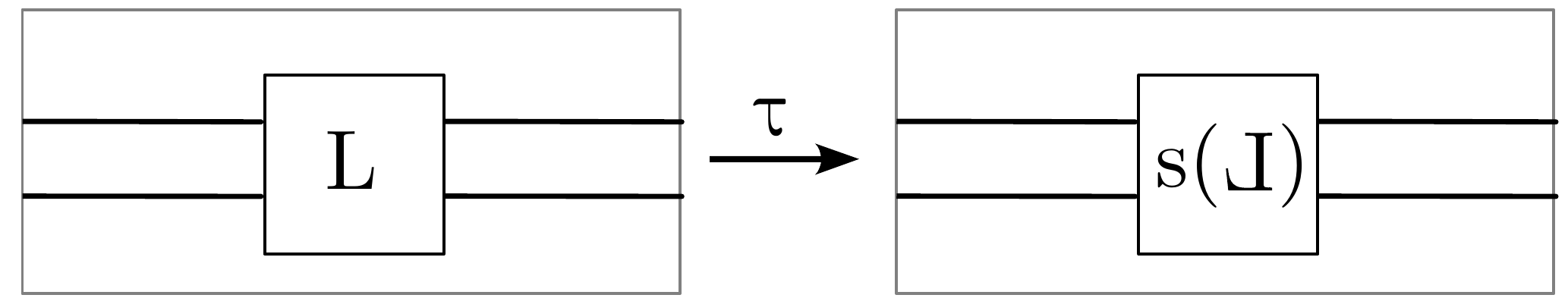}
\caption[legenda elenco figure]{Effect of $\tau$ on a band diagram.}\label{pt}
\end{figure}

For the other two moves it is necessary to use the equivalence between the ball model of the lens space and the genus one Heegaard splitting. First we deform a standard disk diagram into the Heegaard splitting model of the lens space, then we can apply $\sigma_{-}$ or $\sigma_{+}$, the non-isotopic diffeomorphisms of Theorem~\ref{diffeotopy}, and at last we recover a disk diagram of the link from the Heegaard splitting.
The example of Figure~\ref{ps-} describes what happens for $\sigma_{-}: (u,v) \in V_{1} \longleftrightarrow (\bar{u},v) \in V_{2}$, in the case of $L(5,2)$. For $\sigma_{+}$ the description is really similar, except that the diffeomorphism $\sigma_{+}: (u,v) \in V_{1} \longleftrightarrow (u,v) \in V_{2}$ does not mirror $L$. At last, in order to recover a standard disk diagram after the move, it is necessary to apply a sequence of $R_{6}$ moves.
\begin{figure}[h!]          
\center            
\includegraphics[width=12.2cm]{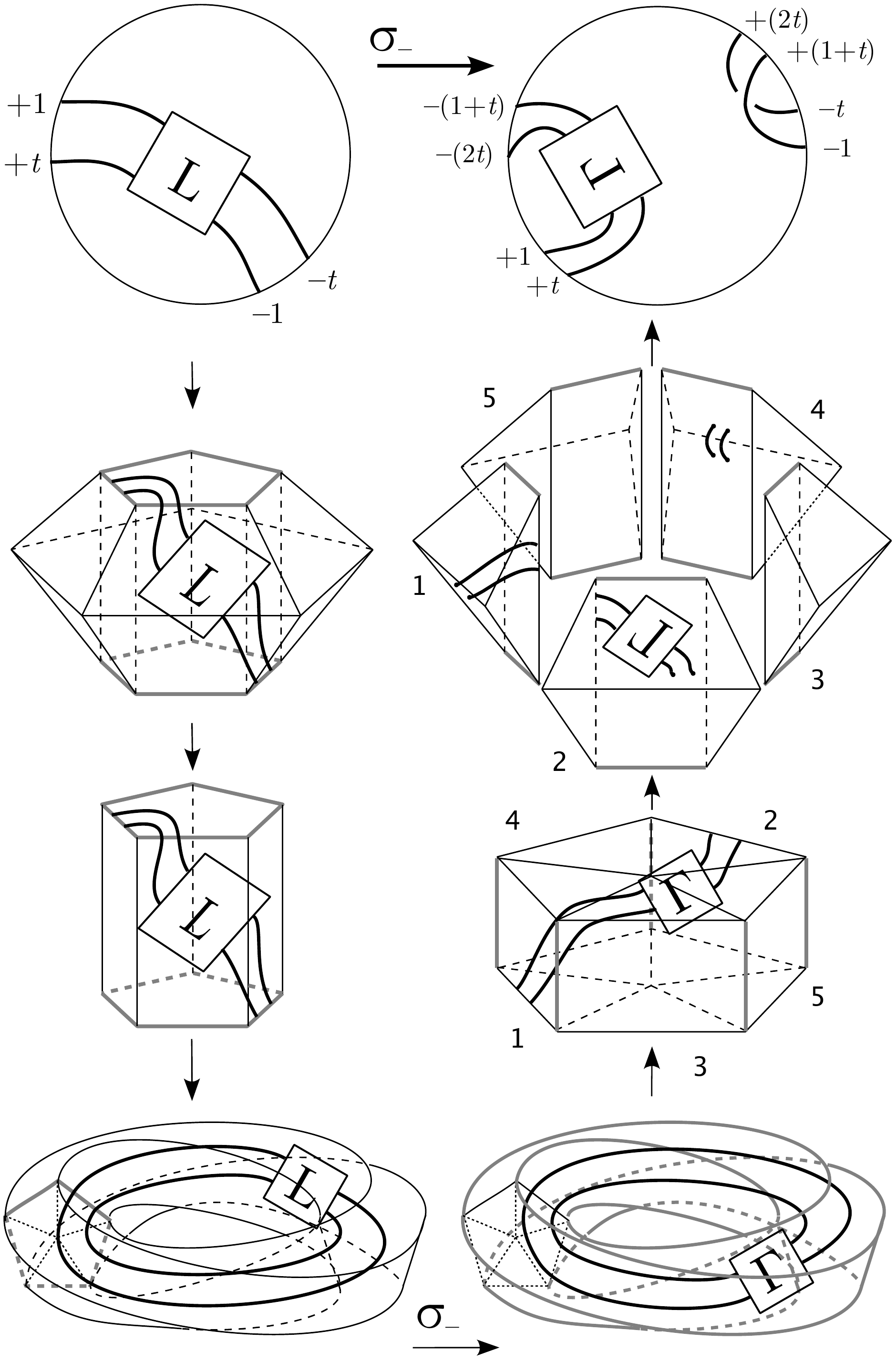}
\caption[legenda elenco figure]{Effect of $\sigma_{-}$ on $L(5,2)$.}\label{ps-}
\end{figure}
\end{proof}

In the case of the projective space, the only effective generator is $\sigma_{-}$ and its effect on disk diagram is the one depicted in Figure \ref{s21-}. For those lens spaces with the element $\sigma_{-}^{3}=\sigma_{-} \circ \tau$ or $\sigma_{+} \circ \tau$ in their diffeotopy group, in order to obtain the effect on a disk diagram of such an element, it is more convenient to apply the $\tau$ move on a standard disk diagram, and then the $\sigma_{+}$ or $\sigma_{-}$ move on the just obtained standard disk diagram.

\begin{figure}[h!]          
\center            
\includegraphics[width=10cm]{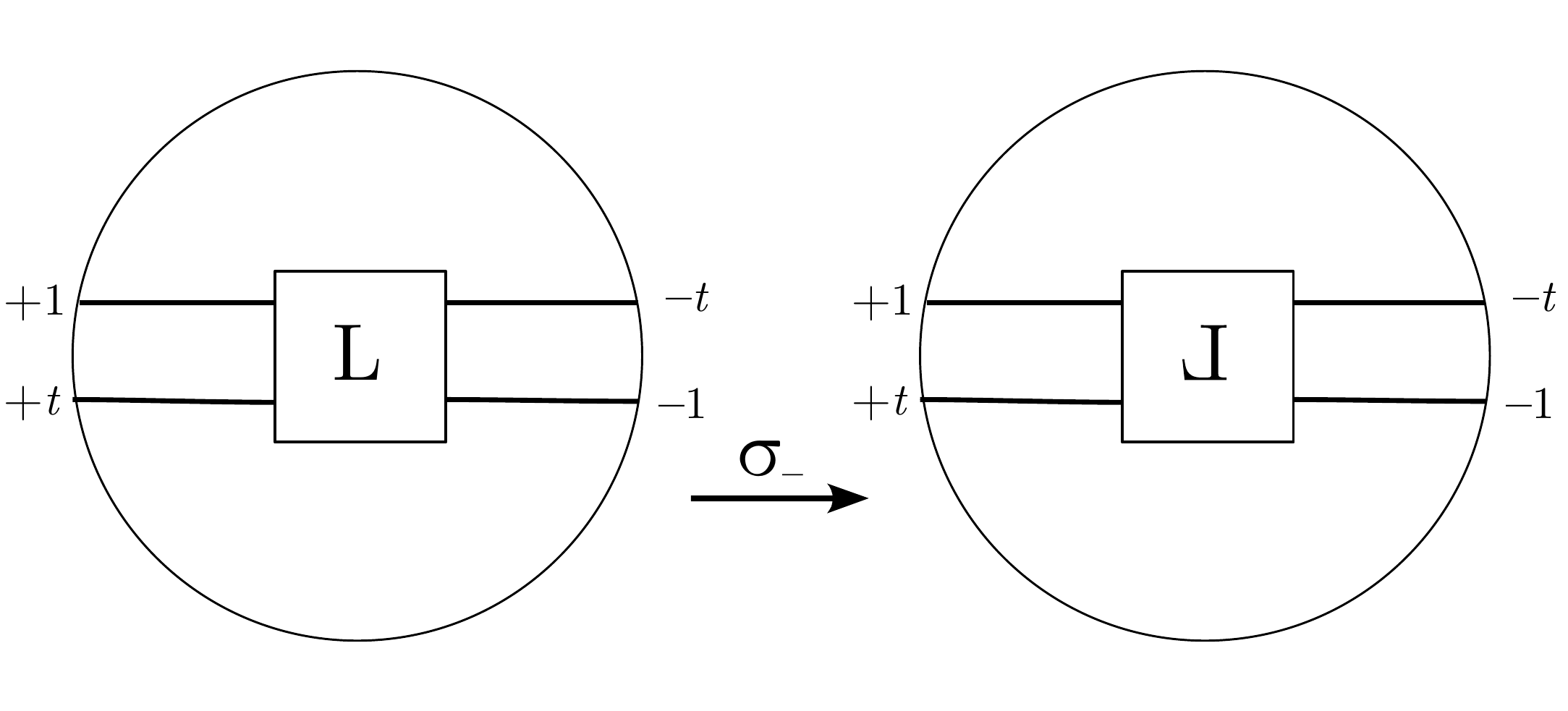}
\caption[legenda elenco figure]{Effect of $\sigma_{-}$ on a disk diagram $\rp3 =L(2,1)$.}\label{s21-}
\end{figure}

\begin{cor}
Two disk diagrams $D$ and $D'$ represent diffeo equivalent links in $L(p,q)$ if and only if they can be connected by a finite sequence of Reidemeister-type moves $R_{1}, \ldots, R_{7}$ and diffeo moves depicted in Figures \ref{t}, \ref{s+} and \ref{s-}.
\end{cor}

\paragraph{Effect of the diffeomorphism moves on the homology class of a knot}

As a consequence of Theorem \ref{dm}, we can easily recover how the diffeomorphisms, and so the diffeo moves change  the homology class of a knot.

\begin{prop}
The effect of the diffeomorphisms on the homology class of a knot $K \subset L(p,q)$ is the following:
\begin{itemize}\itemsep-3pt
\item $[\tau(K)]=-[K]$;
\item $[\sigma_{+}(K)]=q[K]$;
\item $[\sigma_{-}(K)]=q[K]$.
\end{itemize}
\end{prop}

If follows that $[\sigma_{+} \circ \tau (K)]=-q[K]$ and $[\sigma_{-} \circ \tau (K)]=-q[K]$.

\paragraph{Diffeomorphism moves for band diagrams}

Using the transformation between disk and band diagrams (see Figures \ref{LB} and \ref{BL}) developed in  \cite{GM}, we can recover the diffeomorphism moves also on band diagrams. We already described the effect of the $\tau$ move in Figure \ref{pt}, while the effect of $\sigma_{+}$ and $\sigma_{-}$ is depicted in Figures~\ref{bandmove+} and~\ref{bandmove-}.

 \begin{figure}[h!]  
\center            
\includegraphics[width=12cm]{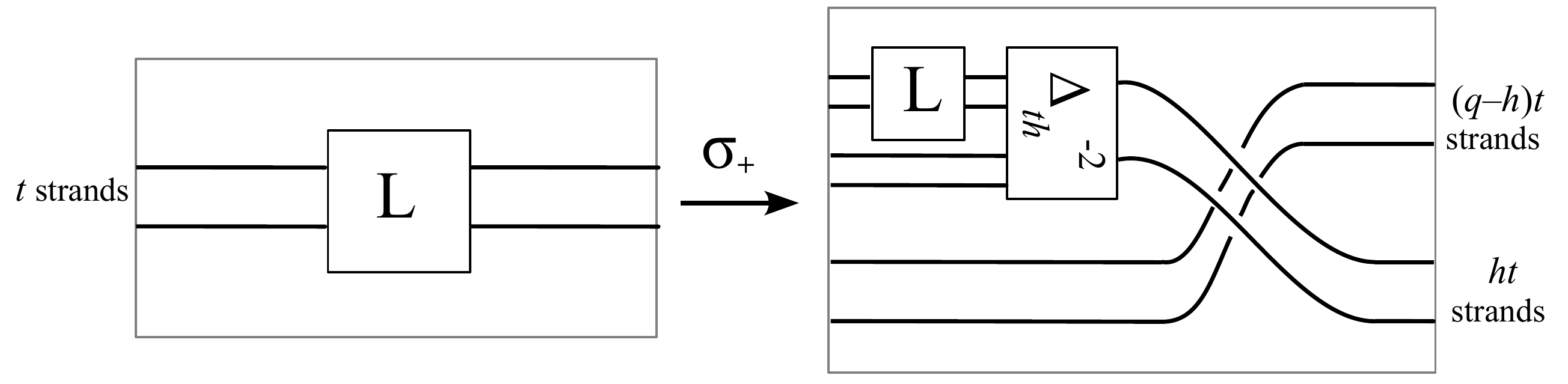}
\caption[legenda elenco figure]{Effect of $\sigma_{+}$ on a band diagram.}\label{bandmove+}
\end{figure}

\begin{figure}[h!]          
\center            
\includegraphics[width=11cm]{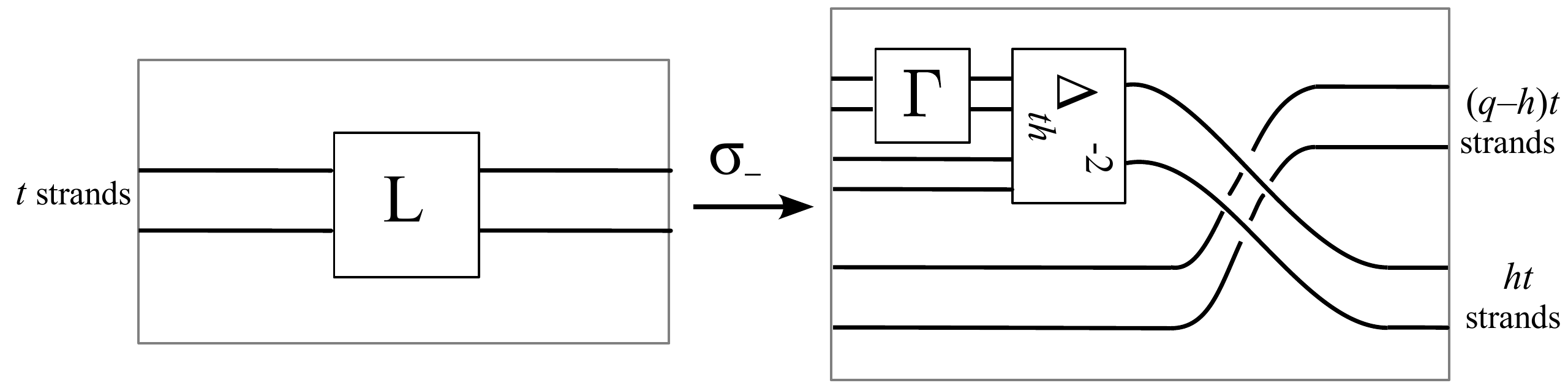}
\caption[legenda elenco figure]{Effect of $\sigma_{-}$ on a band diagram.}\label{bandmove-}
\end{figure}

\paragraph{Diffeomorphism moves for grid diagrams}

As for band diagram, the diffeo moves for grid diagrams can be recovered  using the transition moves between disk and grid diagrams  depicted in Figures \ref{GL} and \ref{LG} and developed in \cite{CMR}. Nevertheless, since the grid rectangle gives rise, after the identifications on the boundary, to the Heegaard torus of the splitting used in the definition of       $\tau$, $\sigma_{+}$ and $\sigma_{-}$, it is very easy to see directly how the three diffeomorphisms  act on it. Indeed,  if  we  number the  columns (resp. row)  from $1$ to $n$  left to right  (resp. up to down)   and the cells inside each column (resp. row) from $1$ to $np$ starting from the top-left one (resp. left one), we obtain the following description:

\begin{itemize}
 \item[$\tau$:] its restriction to each meridian disk is a symmetry along a diameter, so, in the diagram,  we have a left-right exchange inside each box and an up-down exchange along each column;  a marking in the $k$-th cell of the $j$-th column goes into the $(np-k+1)$-th cell of the  $(n-j+1)$-th column (see  Figure \ref{g1});  on the whole, up to a cyclic permutation of the boxes that does not change the represented link, we have a $\pi$ degree rotation around the center of the rectangle (the big dot in the figure);  
 \item[$\sigma_+$:] it exchanges the two solid tori and, consequently,  the corresponding meridian disks,  so, in the diagram, the rows exchange with the columns; a marking contained in the $k$-cell of the $j$-th row goes into the $k$-th cell of the   $j$-th column  (see  Figure \ref{g2}); 
 \item[$\sigma_-$:] it exchanges the meridian disks  of the two solid torus and reflects one of them along a diameter, so, in the diagram, the rows exchange with the columns is followed by a left-right exchange inside each box; a marking contained in the $k$-cell of the $j$-th row goes into the $k$-th cell of the   $(n-j+1)$-th column  (see  Figure \ref{g3});
\end{itemize}

\begin{figure}[h!]  
\center            
\includegraphics[width=9cm]{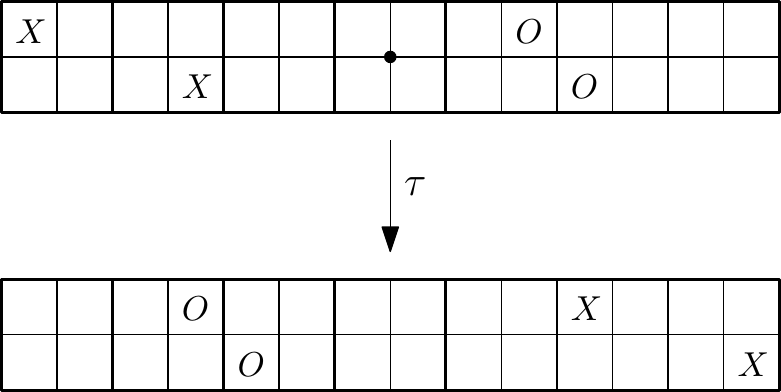}
\caption[legenda elenco figure]{Effect of $\tau$ on a grid diagram in $L(7,2)$.}\label{g1}
\end{figure}

\begin{figure}[h!]  
\center            
\includegraphics[width=9cm]{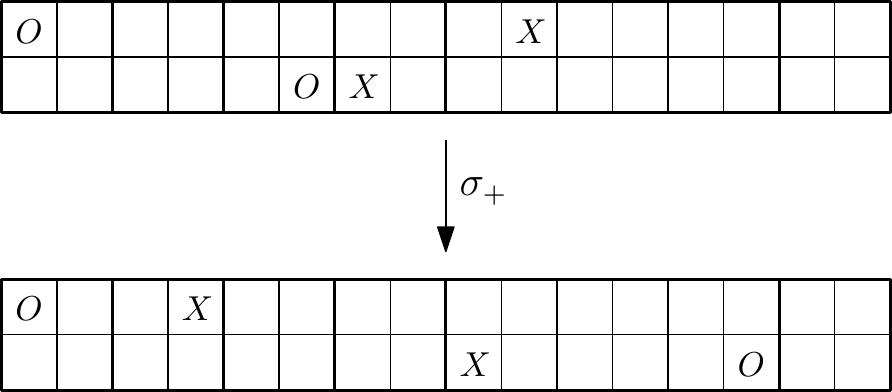}
\caption[legenda elenco figure]{Effect of $\sigma_+$ on a grid diagram in $L(8,3)$.}\label{g2}
\end{figure}

\begin{figure}[h!]  
\center            
\includegraphics[width=9cm]{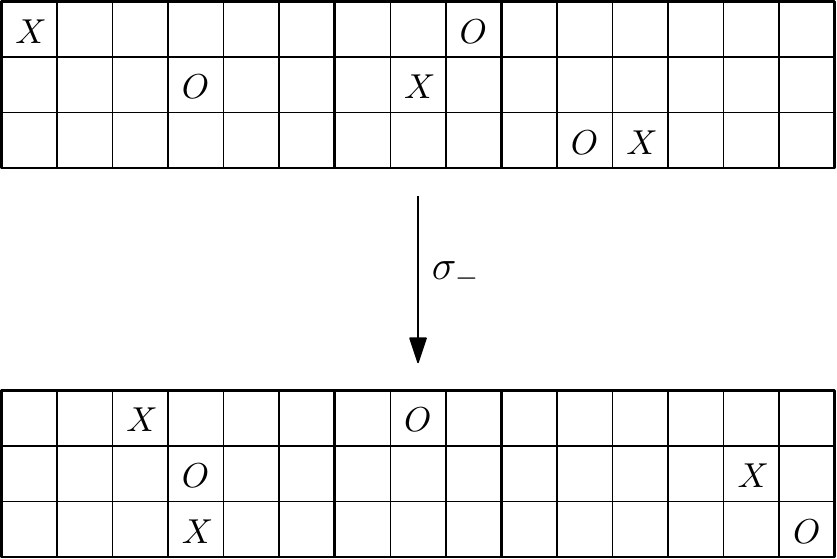}
\caption[legenda elenco figure]{Effect of $\sigma_{-}$ on a grid diagram in $L(5,2)$.}\label{g3}
\end{figure}


\section{Lift}

The lift of a link $L \subset L(p,q)$ is the pre-image under the cyclic covering $p: \s3 \to L(p,q)$, that is $p^{-1}(L)=\widetilde{L}$. The lift is a strong isotopy invariant of links in lens spaces, and using the knowledge of links in the $3$-sphere it is possible to distinguish a wide range of links in lens spaces. In this section we will discuss when the lift can distinguish links in lens spaces, both up to isotopy and up diffeomorphism.

\paragraph{Effect of the diffeomorphism moves on the lift}
Before examining how the lift changes under  the diffeomorphism moves, we have to set notations and recall some important facts.

First of all,  if  $L_{1}, \ldots, L_{\nu}$ are the components of  a link $L\subset L(p,q)$, and $\delta_{i}=[L_{i}]$ is  the homology class of  $L_i$, the number of components of  the lift $\widetilde{L}$ is $ \sum_{i=1, \ldots, \nu} \gcd (\delta_{i},p)$.
A knot in $L(p,q)$ that lifts to a knot (i.e., having only one component) is said to be \textit{primitive-homologous}, since its homology class is co-prime with $p$. Moreover, for each kind of representation, a  description of how to find a diagram for  $\widetilde{L}$ starting from that of  $L$  is given: see  \cite{M1} for disk diagrams,  \cite{BGH} for grid diagrams and   \cite{GM} for band diagrams.


If $L$ is a link in $\s3$, let $m(L)$ be the mirror image of $L$. Clearly $L$ and $m(L)$ are diffeo equivalent in $\s3$ but not always isotopy equivalent: if they are also isotopy equivalent $L$ is said to be \textit{amphicheiral}.

\begin{prop}\label{liftmoves}
The effects of the diffeomorphisms on the lift of a link \mbox{$L \subset L(p,q)$} are the following:
\begin{itemize}\itemsep-3pt
\item $\widetilde{\tau(L)}=\widetilde{L}$;
\item $\widetilde{\sigma_{+}(L)}=\widetilde{L}$;
\item $\widetilde{\sigma_{-}(L)}=m(\widetilde{L})$.
\end{itemize}
\end{prop}

\begin{proof}
The ambient isotopies of the link equivalence are preserved when they lift from $L(p,q)$ to $\s3$, so the statement is immediate if we consider how the diffeomorphism of $L(p,q)$ lift to $\s3$: $\tau$ and $\sigma_{+}$ are orientation preserving and lift to the identity, while $\sigma_{-}$ is orientation reversing and lifts to the orientation reversing diffeomorphism of $\s3$.
\end{proof}

\paragraph{Completeness of the lift}

Since the lift of a link in a lens space comes from a cyclic covering,  a natural question to ask is whether it is a complete invariant, at least for some family of knots. We analyze the case of oriented primitive-homologous knots when they are considered up to diffeo equivalence, using the following theorem of Sakuma, also proved by Boileau and Flapan, about freely periodic knots.

A knot $K$ in $\s3$ is said to be freely periodic if there is a free finite cyclic action on $\s3$ that fix $K$. Clearly the quotient under this action is a lens space with a knot inside it that lifts to $K$.
In this case, if $\textrm{Diff}^{*}(\s3,K)$ denotes the group of diffeomorphisms of the pair $(\s3,K)$, which preserves the orientation of both $\s3$ and $K$, up to isotopy, we say that a \textit{symmetry} $G$ of a knot $K$ in $\s3$ is a finite subgroup of $\textrm{Diff}^{*}(\s3,K)$, up to conjugation.

\begin{teo}{\upshape{\cite{BF,Sa}}}\label{hompair}
Suppose that a knot \mbox{$K\subset \s3$} has free period $p$. Then there is a unique symmetry $G$ of $K$ realizing it, provided that (i) $K$ is prime, or (ii) $K$ is composite and the slope is specified.
\end{teo}

If we translate this theorem into the language of knots in lens spaces, specifying  the slope is equivalent to fixing the parameter $q$ of the lens space. As a consequence, two primitive-homologous knots $K_1$ and $K_2$ in $L(p,q)$ with diffeo equivalent non-trivial lifts are necessarily diffeo equivalent in $L(p,q)$. 

Combining Proposition~\ref{liftmoves} and Theorem~\ref{hompair} it is possible to obtain a similar statement up to isotopy.

\begin{teo}\label{liftteo}
Let $K_{1}$ and $K_{2}$ be two oriented primitive-homologous knots in the lens space $L(p,q)$ such that they have isotopy equivalent lifts  not being the trivial knot.
In the case that their lift is amphicheiral:
\begin{itemize}\itemsep-3pt
\item if $p=2$, then $K_{2}$ is isotopy equivalent either to $K_{1}$ or to $\sigma_{-}(K_{1})$; 
\item if  $q \equiv \pm 1 \mod p$ and $p \neq 2$, then $K_{2}$ is isotopy equivalent either to $K_{1}$ or to $\tau(K_{1})$;
\item if  $q^2 \equiv + 1 \mod p$ and $q \not \equiv \pm 1 \mod p$, then $K_{2}$ is isotopy equivalent to one of the following: $K_{1}$, $\tau(K_{1})$, $\sigma_{+}(K_{1})$, $\tau(\sigma_{+}(K_{1}))$; 
\item if $q^2 \equiv - 1 \mod p$ and $p \neq 2$, then $K_{2}$ is isotopy equivalent to one of the following: $K_{1}$, $\sigma_{-}(K_{1})$, $\tau(K_{1})$, $\sigma_{-}(\tau(K_{1}))$; 
\item if $q^{2} \not \equiv \pm 1 \mod p$, then $K_{2}$ is isotopy equivalent either to $K_{1}$ or to $\tau(K_{1})$.
\end{itemize}

In the case that their lift is not amphicheiral:

\begin{itemize}\itemsep-3pt
\item if $p=2$, then $K_{2}$ is isotopy equivalent to $K_{1}$; 
\item if  $q \equiv \pm 1 \mod p$ and $p \neq 2$, then $K_{2}$ is isotopy equivalent either to $K_{1}$ or to $\tau(K_{1})$;
\item if  $q^2 \equiv + 1 \mod p$ and $q \not \equiv \pm 1 \mod p$, then $K_{2}$ is isotopy equivalent to one of the following: $K_{1}$, $\tau(K_{1})$, $\sigma_{+}(K_{1})$, $\tau(\sigma_{+}(K_{1}))$; 
\item if $q^2 \equiv - 1 \mod p$ and $p \neq 2$, then $K_{2}$ is isotopy equivalent either to $K_{1}$ or to $\tau(K_{1})$; 
\item if $q^{2} \not \equiv \pm 1 \mod p$, then $K_{2}$ is isotopy equivalent either to $K_{1}$ or to $\tau(K_{1})$.
\end{itemize}
\end{teo}
\begin{proof}
Let us analyze just the case of $L(2,1)$ since the other cases can be treated similarly.
If $\widetilde{K_{1}}\cong \widetilde{K_{2}}$, then by Theorem~\ref{hompair} we can assume that $K_{1}$ and $K_{2}$ are diffeo equivalent, that is, by Proposition~\ref{dm}, $K_{2}$ is isotopy equivalent either to $K_{1}$ or to $\sigma_{-}(K_{1})$. Moreover, if $\widetilde{K_{1}}\cong \widetilde{K_{2}}$ is not amphicheiral, that is to say, not equivalent to its mirror image, by Proposition~\ref{liftmoves} we know that $\widetilde{\sigma_{-}(K_{1})}=m(\widetilde{K_{1}})=m(\widetilde{K_{2}})$ and thus $K_{2}$ is necessarily isotopy equivalent to $K_{1}$.
\end{proof}

\vspace{2mm}
\textit{Acknowledgments:} The authors would like to thank  Michele Mulazzani for its useful suggestions.


\vspace{15 pt} {ALESSIA CATTABRIGA, Department of Mathematics,
University of Bologna, ITALY. E-mail: alessia.cattabriga@unibo.it}

\vspace{15 pt} {ENRICO MANFREDI, Department of Mathematics,
University of Bologna, ITALY. E-mail: enrico.manfredi3@unibo.it}


\end{document}